%% file: ms.tex
\documentclass{siamart1116}
\pdfoutput=1
\usepackage{amsfonts}
\usepackage{mathtools}
\usepackage{enumerate}
\usepackage{dsfont}
\usepackage{accents}
\usepackage{cite}

\ifpdf
  \DeclareGraphicsExtensions{.eps,.pdf,.png,.jpg}
\else
  \DeclareGraphicsExtensions{.eps}
\fi

\input{slidesymbols}
\input{macrosKevin}

\newcommand{\orange}[1]{{#1}}

\newtheorem{remark}[theorem]{Remark}
\newtheorem{assumption}[theorem]{Assumption}
\newtheorem{example}[theorem]{Example}
\newtheorem{numexp}[theorem]{Numerical Experiment}

\newcommand{\fractx}[2]{{\textstyle\frac{#1}{#2}}}

\numberwithin{theorem}{section}

\newcommand{\TheTitle}{Weakly-normal basis vector fields in RKHS with an application to shape Newton methods}
\newcommand{\TheShortTitle}{Weakly-normal basis vector fields in RKHS}
\newcommand{\TheAuthors}{A.~Paganini and K.~Sturm}

\headers{\TheShortTitle}{\TheAuthors}

\title{{\TheTitle}\thanks{Submitted to the editors DATE.
}}

\author{
  Alberto Paganini
  \thanks{University of Oxford \email{paganini@maths.ox.ac.uk}}
  \and
  Kevin Sturm
 \thanks{Johann Radon Institute, \email{kevin.sturm@oeaw.ac.at}}
  }

\usepackage{amsopn}

\ifpdf
\hypersetup{
  pdftitle={\TheTitle},
  pdfauthor={\TheAuthors}
}
\fi

\begin{document}

\maketitle

% REQUIRED
\begin{abstract}
We construct a space of vector fields that are normal to differentiable curves in the plane.
Its basis functions are defined via saddle point variational problems in
reproducing kernel Hilbert spaces (RKHSs).
First, we study the properties of these basis vector fields and show how to approximate them.
Then, we employ this basis to discretise shape Newton methods and
investigate the impact of this discretisation on convergence rates.
\end{abstract}

% REQUIRED

% REQUIRED

\section{Introduction}

%Shape optimisation is an active field of research that studies how
%to construct a shape $\Omega$ such that a shape functional $J$ is minimised. 
%An important tool to tackle shape optimisation problems are so-called shape derivatives.
{
Shape optimisation %is a thriving field of research and
studies how
to design a domain $\Omega$ such that a shape functional $J$ is minimised.
Shape optimization problems arise naturally in numerous 
industrial applications. Just to mention a few examples, shape optimization can
be used to improve the design of electrodes \cite{BaCiOfStZa15}, bridges \cite{AlDaFr14}, breakwaters \cite{KeKr16},
optical lenses \cite{PaHaSaHa15}, airplane components \cite{CoRiBe15}, and 3D printers \cite{DaEsFaMi17}.
In light of its wide range of applications, it is not a surprise that several commercial
products offer shape optimisation features.
}

{
Shape optimisation problems are inherently infinite dimensional.
Therefore, any numerical shape optimisation algorithm must rely on a discretisation of the
space of shapes.
%The second key characteristic is that 
Additionally, shape optimization problems are in most instances nonconvex.
Therefore, numerical shape optimisation software commonly relies on iterative
optimisation algorithms.
%Nonconvex optimization problems are commonly solved iteratively.
In particular,
in light of the large number of optimization parameters, it is advantageous to use
optimization strategies that involve derivatives of the functional $J$, like steepest descent
or Newton's algorithm.

The derivative of $J$ with respect to perturbations of $\Omega$ is called its shape derivative
and its technical definition depends on how these domain perturbations are modeled. 
Historically, domain perturbations have been defined in terms of velocity flows or perturbations of the identity \cite{delfourZolesioB,HenrotPierre05, SokolowskiZolesio1}.
Recently, Allaire et al. \cite{AlCaVi16} have introduced domain perturbations
in terms of Hamilton-Jacobi equations. 
These definitions lead to
the same formulas for first-order shape derivatives, but differ for second-order ones.
The equivalence of the three approaches for the first shape-derivative stems from the fact
that each approach constructs domain perturbations from vector fields, and the
Taylor series expansions of these constructions agree to first order.
However, the second-order terms of these series do not match. Therefore,
the corresponding shape Newton methods potentially deliver different shape updates.

Regardless of the definition used, both first- and second-order shape derivatives have a
nontrivial kernel. {For instance, shape derivatives are zero in the direction of vector fields
that vanish on $\partial\Omega$.}
The Hadamard-Zol\'esio theorem and its second-order counterpart \cite{a_NOPI_2002a}
provide a complete characterization of the kernel of shape derivatives.
This information needs to be taken into account to prove convergence rates
of shape Newton methods \cite{sturm3, Schulz1} {(because one must
``quotient'' the kernel from the search space)} and has a direct consequence
for numerical computations: if the discretisation of domain perturbations is not
chosen carefully, the matrix that results by restricting the second-order shape
derivative to the discrete trial space may fail to be positive definite
(it is usually positive semi-definite). In particular, the basis vector fields used to discretise
domain perturbations must necessarily have nonvanishing normal component on the
boundary of $\Omega$.

%Historically, shape derivatives \cite{delfourZolesioB,HenrotPierre05, SokolowskiZolesio1}
%have been defined by flows of vector fields or by perturbations of the identity.
%More recently, Allaire et al. \cite{AlCaVi16} have introduced an alternative definition of shape derivatives
%in terms of Hamilton-Jacobi equations, which coincides with previous derivatives when restricted to normal vector fields.
%The flow and perturbation of identity approach definitions lead to equivalent first derivatives but differing second %ones \cite{delfourZolesioB}. 
%However, not all second derivates are suited to formulate shape Newton methods
%\cite{sturm3, Sc17}.

%Several papers propose Newton-type methods to specific PDE constrained problems with specific domain representations, e.g., star shaped domains; \cite{EPHA1,EP1,EPHASH1,SCSIWE1,SCSIWE2}.
%In \cite{Schulz1}, Schulz proves that, under certain assumptions, a Newton method formulated in shape spaces \cite{RingWirth1} converges quadratically, whereas
%superlinear convergence of a (discrete) shape Newton method
%in Micheletti spaces has been showed in \cite{sturm3}.
%In both papers, the authors have to tackle the problem of
%the shape Newton equation being well-posed only with respect to normal perturbations of the boundary.
{Our goal is to construct basis vector fields that are normal to a given boundary
$\partial\Omega$ and that have good approximation properties.}
The literature contains articles that address this issue and rely
on vector fields with this geometric feature.
In \cite{EPHA1,EP1,EPHASH1}, the authors consider star-shaped domains whose parametrization
formula is known analytically and construct the basis vector fields using this parametrization.
This approach is at the same time ideal for and limited to shape optimization problems
for which all computations can be restricted to the boundary $\partial\Omega$,
because the basis vector fields constructed in \cite{EPHA1,EP1,EPHASH1}
are defined only on $\partial\Omega$.
{To restrict all computations to the boundary, the domain
$\Omega$ must satisfy certain regularity requirements, and if the shape functional
$J$ is constrained to a boundary value problem, its solution must also be sufficiently
smooth. If these conditions are not fulfilled,}
%If 
restricting all computations to $\partial\Omega$ is not possible, and
one needs to employ basis vector fields whose support has
nonzero measure. In practice, it is not difficult to construct such vector fields replacing
$\Omega$ with a triangulation and using finite elements \cite{SCSIWE1,SCSIWE2}.
However, the approximation properties of the constructed space are unclear,
and it is difficult to understand whether the Newton updates are well-defined
in the zero mesh-width limit. 

In our work, we use reproducing kernel Hilbert spaces (RKHSs)
to define a class of basis vector fields that are \emph{weakly-normal} to the boundary of a domain $\Omega$
and whose support has nonzero measure.
We use the term weakly-normal because this geometric property is enforced with the help of
saddle point variational problems. This approach allows us to ensure that the resulting
vector fields are linearly independent and normal to the shape boundary $\partial \Omega$.
Therefore, they can be used to formulate discrete shape Newton methods.
Moreover, these vector fields span a dense subset of the space of normal vector fields,
and it is possible to analyse the error between exact and approximate shape Newton updates.

For simplicity, we restrict the exposition to the two dimensional case.
However, the approach can be adapted to higher dimensions,
although the implementation aspects may become more challenging.

}
%\paragraph{Structure of the paper}
%The paper is organised in two parts.
%
%\Cref{sec:normalVFs} deals with the formulation and the analysis of the proposed
%new class of {weakly-}normal vector fields.
%We study their properties, explain how to approximate them, and analyse the error of this approximation.
%
%\Cref{sec:Newton} deals with a shape {optimization test case, its shape Newton method, and
%the corresponding discretisation based on} %Newton methodand to its discretisation with
%the basis functions introduced in \Cref{sec:normalVFs}.
%%For a particular test case, we
%{We }prove quasi-optimality of the approximate
%shape Newton update {and investigate }%. Moreover, we examine numerically
%the impact of this approximation on the convergence rate of the optimization algorithm.
%
%The theoretical results of this work are supported by numerical experiments performed in \textsc{Matlab}
%and partly based on the Chebfun library \cite{chebfun}.

\section{Normal vector fields in RKHS}\label{sec:normalVFs}

\subsection{RKHSs, Cartesian products, and trace spaces}
We begin with some basic definitions of reproducing kernels.
\begin{definition}
	Let $\mathcal X\subset \VR^2$ be an arbitrary set. A function 
	$\ksf:\mathcal \Cx\times\Cx\to \VR$ is called 
	\emph{reproducing kernel} for the Hilbert space 
	$\left(\Ch(\mathcal X),(\cdot,\cdot)_\Ch\right)$ of functions $f:\Cx\to \VR$, if
	\begin{enumerate}[$(a)$]
		\item 
		$\ksf(\Vx,\cdot)\in \Ch(\Cx)$ for every $\Vx\in \mathcal X$,
		\item
		$(\ksf(\Vx,\cdot), f)_{\Ch} = f(\Vx)	$ for every $\Vx\in \mathcal X$ and every $f\in \Ch(\Cx)$.
	\end{enumerate}
\end{definition}

\begin{definition}
	A kernel $\ksf:\Cx\times \Cx\to \VR$ is said to be positive-definite on $\Cx$ if,
	for every finite subset of pairwise distinct points $\{\Vx_i\}_{i=1}^n\subset\Cx$,
	the matrix 	$(\ksf(\Vx_i,\Vx_j))_{i,j=1}^{n}$ is positive-definite.
\end{definition}

\begin{definition}
	A kernel $\ksf:\Cx\times \Cx\to \VR$ is said to be symmetric if
	\begin{equation*} \ksf(\Vx, \Vy) = \ksf(\Vy, \Vx)  \quad \text{for all } \Vx, \Vy \in \Cx\,.\end{equation*}
\end{definition}

The main reason to consider symmetric positive-definite reproducing kernels is that they
can be used to define certain Hilbert Spaces.

\begin{theorem}
To a positive-definite and symmetric kernel $\ksf:\Cx\times \Cx \to \VR$ defined on an arbitrary set
$\Cx\subset \VR^2$ corresponds a unique reproducing kernel Hilbert space  $(\Ch(\Cx),(\cdot,\cdot)_{\Ch})$ (RKHS),
which is called \emph{native space} of $\ksf$. Note that $(\cdot,\cdot)_{\Ch}$ depends on $\ksf$. 
\end{theorem}

There is a variety of symmetric positive-definite reproducing kernels \cite{Wendlandbook}
(each tailored to specific applications).
In this work, we are particularly interested in Wendland kernels with compact support
because, for applications in shape optimisation, compact support is a useful 
%(and sometimes necessary \cite{LaSt16}) feature for
\orange{(it implies sparsity of the
discretised shape Hessian) and sometimes necessary (when part of the
domain is not ``free to move'' \cite{LaSt16})} feature .
In the next example, we list three Wendland kernels and recall their native spaces.

\begin{example}\label{ex:kernel} Classical examples of positive-definite kernels
can be found in \cite{Wendlandbook}. In this work, we are particularly interested in
\begin{equation}\label{eq:k4}
\ksf^\sigma_4(x,y)\coloneqq \left(1-\frac{|x-y|}{\sigma}\right)_+^4\left(4\frac{|x-y|}{\sigma} +1\right), \quad \sigma>0,
\end{equation}
where
\begin{equation*}(\cdot)_+:\VR\to\VR^+_0\,,\quad x\mapsto (x)_+\coloneqq\max(0, x)\,.\end{equation*}
When $\sigma=1$,  \cref{eq:k4} is the lowest order Wendland
kernel with compact support that is
positive-definite on $\VR^d$ (for $d\le 3$) and is of class $C^2$.
When $\Cx=\VR^2$, its native space
is $H^{2.5}(\VR^2)$ \cite[p. 160, Thm. 10.35]{Wendlandbook}. 
When $\Cx=\Omega$ is a bounded smooth domain, its native space is $H^{2.5}(\Omega)$
(because every function $f\in\Ch(\Omega)$ can be extended to a function
$\tilde f\in\Ch(\VR^2) = H^{2.5}(\VR^2)$ \cite[p. 169, Thm. 10.46]{Wendlandbook}).
Similar properties can be proved for the $C^4$-kernel $\ksf^\sigma_6$ and the $C^6$-kernel $\ksf^\sigma_8$,
which are given by \cite{Wendlandbook},
\begin{align}
\label{eq:k6}
\ksf^\sigma_6(x,y)&\coloneqq 
\frac{1}{3}\left(1-\frac{|x-y|}{\sigma}\right)_+^6\left(35\frac{|x-y|^2}{\sigma^2}+18\frac{|x-y|}{\sigma} +3\right), \\
\label{eq:k8}
\ksf^\sigma_8(x,y)&\coloneqq \left(1-\frac{|x-y|}{\sigma}\right)_+^8\left(32\frac{|x-y|^3}{\sigma^3}+25\frac{|x-y|^2}{\sigma^2}
+8\frac{|x-y|}{\sigma} +1\right)\,,
\end{align}
respectively. Their corresponding native spaces in dimension two are $H^{3.5}(\VR^2)$ and 
$H^{4.5}(\VR^2)$, respectively \cite[p. 160, Thm. 10.35]{Wendlandbook}.
\end{example}

Scalar reproducing kernels can be used to define Hilbert spaces of vector valued functions, too \cite[Lemma 3.6]{a_EIST_2017a}.
\begin{proposition}
The Cartesian product $[\Ch(\Cx)]^2\coloneqq \Ch(\Cx)\times\Ch(\Cx)$ of a RKHS
$\Ch(\Cx)$ (with reproducing kernel $\ksf$) is itself a RKHS. Its (matrix valued) reproducing kernel is $\ksf(\cdot,\cdot) \VI$, where $\VI\in\VR^2$ denotes the identity matrix. {Henceforth,
we denote by $(\cdot,\cdot)_{[\Ch]^2}$ the (canonical) inner product of $[\Ch(\Cx)]^2$.}
\end{proposition}

Next, we describe how to construct RKHSs on boundaries of bounded subdomains of $\VR^2$.
Let $\ksf$ be a positive-definite kernel on $\VR^2$, and let $\Ch(\VR^2)$ be 
its native space. Let $\Omega\subset \VR^2$ be a nonempty,
bounded, open and smooth set (e.g. $C^1$). 
We denote by $\nubf$ and $\taubf$ the unit normal and unit tangential vector fields along $\partial \Omega$. 

The restriction of $\ksf$ onto $\partial \Omega$ defines a positive-definite kernel
$\ksf\vert_{\partial\Omega}$, which
uniquely characterises the RKHS \cite[p.169]{Wendlandbook}
\begin{equation}
\Ch(\partial \Omega) = \{f|_{\partial \Omega}:\; f\in \Ch(\VR^2)\}\,.
\end{equation}
Similarly, it is possible to define the space of vector fields restricted onto $\partial\Omega$ by
\begin{equation}
[\mathcal H(\partial \Omega)]^2 = \{\Vf|_{\partial \Omega}:\; \Vf\in [\mathcal H(\VR^2)]^2\}.
\end{equation}

The next example identifies the space $\mathcal H(\partial \Omega)$ for the kernel \cref{eq:k4}.
\begin{example}
The native space of the restriction onto $\partial\Omega$ of the Wendland kernel $\ksf_4^\sigma$
(with $\sigma$ fixed)
is $H^2(\partial \Omega)$. This follows from $\Ch(\VR^2) = H^{2.5}(\VR^2)$ and 
the standard Sobolev trace theorem.
\end{example}
We conclude this section by defining the RKHS of normal and tangential vector fields:
\begin{align}\label{eq:Hnu}
[\Ch(\partial \Omega)]^2_{\nubf} &\coloneqq \{\Vf\in[\mathcal H(\VR^2)]^2:\;\Vf\cdot \taubf =0 \quad \text{ on } \partial \Omega  \}\,,\\
[\Ch(\partial \Omega)]^2_{\taubf} &\coloneqq \{\Vf\in[\mathcal H(\VR^2)]^2:\; \Vf\cdot \nubf =0 \quad \text{ on } \partial \Omega  \}.
\end{align}

\subsection{Weakly-normal basis functions}
In this section, we introduce a novel class of normal vector fields.
We begin with the following assumption on the reproducing kernel $\ksf$.

\begin{assumption}\label{ass:one}
Let $k\geq 1$ be an integer and let $\partial \Omega$ be of class $C^{k+1}$. We assume that
$\ksf:\VR^2\times \VR^2 \to \VR$ is a symmetric positive-definite kernel on $\VR^2$
with the property \begin{equation}\label{eq:nomalizedk}
\ksf(\Vx,\Vx)=1 \quad \text{for all } \Vx\in \partial\Omega\,.
\end{equation}
We assume that for every $\Vx\in \VR^2$ we have $\ksf(\Vx,\cdot) \in C^k(\VR^2)$.
Furthermore, we assume that there is a constant $c_{\Omega}>0$, 
such that for every $\phi\in C^k(\partial\Omega)$ and every $f \in \Ch(\partial \Omega)$,
\begin{equation}\label{eq:module}
\phi f\in \Ch(\partial\Omega)\quad \text{and}\quad \|\phi f\|_{\Ch} \le c_{\Omega} \|\phi\|_{C^k(\partial\Omega)}\|f\|_{\Ch} .
\end{equation}

Finally, we assume that the {constant function} %identity map 
$1:\partial \Omega \to\VR$ % \partial \Omega$
belongs to $\Ch(\partial \Omega)$ %$[\Ch(\partial \Omega)]^2$
and, for simplicity, that the support of $\ksf$ is connected.

\end{assumption}
\begin{example}
The kernels $\ksf^\sigma_4, \ksf^\sigma_6$ and $\ksf^\sigma_8$ satisfy \cref{ass:one} for $k=2,3,4$, respectively. 
\end{example}

\begin{remark}\label{rmk:directsum}
Property \cref{eq:module} implies that 
 $C^k(\partial\Omega)\subset\Ch(\partial\Omega)$ (by \cref{eq:module}), and that
$[\Ch(\partial \Omega)]^2 = [\Ch(\partial \Omega)]^2_{\nubf}\oplus [\Ch(\partial \Omega)]^2_{\taubf}$
(because every vector field $\Vv \in [\Ch(\partial \Omega)]^2$ can be decomposed into
$ \Vv = (\Vv \cdot \nubf)\nubf + (\Vv \cdot \taubf) \taubf$ on $\partial \Omega$).
\end{remark}

\begin{definition}\label{def:rx}
For a point $\Vx\in\partial\Omega$, let $(\Vr_\Vx, p_\Vx)\in [\Ch(\partial\Omega)]^2\times \Ch(\partial\Omega)$
be the solution of
\begin{subequations}
\label{eq:saddlepoint}
\begin{align}
(\Vr_\Vx,\varphibf)_{[\Ch]^2}  +(\varphibf \cdot \taubf, p_\Vx )_{\mathcal H}
&= \nubf(\Vx)\cdot \varphibf(\Vx)&& \text{ for all } \varphibf \in [\mathcal H(\partial\Omega)]^2,\label{eq:1RKHS2}\\
\;(\Vr_\Vx\cdot \taubf, \psi)_{\mathcal H}
& =0 &&\text{ for all } \psi \in\mathcal H(\partial\Omega) \label{eq:2RKHS2}\,.
\end{align}
\end{subequations}
The function $\Vr_\Vx$ is the \emph{weakly-normal basis function} associated with $\Vx$.
\end{definition}

{To show that \cref{def:rx} makes sense, we need to prove
that \cref{eq:saddlepoint} admits a unique and stable solution. This is done in
\Cref{lem:rxwelldefined}, which in turn relies on \cref{lem:B_surjective}. Before showing
\cref{lem:B_surjective}, we introduce the following notation:}
henceforth we denote by $\Ve_i, i =1,2,$ the canonical basis of $\VR^2$ and
use the norm
\begin{equation}
\|\Vf\|_{C^k}\coloneqq \sqrt{\Vert \Vf\cdot\Ve_1\Vert_{C^k}^2+\Vert \Vf\cdot\Ve_2\Vert_{C^k}^2}
\quad \text{for all } \Vf\in C^k(\VR^2, \VR^2)\,.
\end{equation}

\begin{lemma}\label{lem:B_surjective}
The operator $B:[\mathcal H(\partial\Omega)]^2 \rightarrow \mathcal H(\partial\Omega)^*$ defined by  $\varphibf \mapsto (\varphibf\cdot \taubf, \cdot)_{\mathcal H}$ is surjective. Moreover, 
\begin{equation}\label{eq:coercivity_Bt}
\|B^*\psi\|_{ ([\Ch]^2)^*}%\Ch^*}
 \ge \frac{\|\psi\|_{\Ch}}{c_{\Omega} \|\taubf\|_{C^k}}  \quad \text{ for all } \psi \in \Ch(\partial\Omega)\,.
\end{equation} 
\end{lemma}

\begin{proof}
{To show that $B$ is surjective,} let $f\in \Ch(\partial\Omega)^*$. By the Riesz representation theorem, there is a function
$\eta_f\in \Ch(\partial\Omega)$ that satisfies $f(\psi) = (\eta_f, \psi)_\Ch$ for all $\psi\in \Ch(\partial\Omega)$.
Since the function $\varphibf \coloneqq \eta_f\taubf\in [\mathcal H(\partial\Omega)]^2$ satisfies $B\varphibf = f$, the operator $B$ is surjective. To verify \cref{eq:coercivity_Bt}, first note that
inequality \cref{eq:module} implies $\|\taubf\psi\|_{[\Ch]^2}\leq c_{\Omega} \|\taubf\|_{C^k} \|\psi\|_{\Ch}$
for every $\psi \in \Ch(\partial\Omega)$. Thus,
\begin{align*}
\|B^*\psi\|_{ ([\Ch]^2)^*}%\Ch^*}
= \sup_{\substack{\varphibf\in [\Ch(\partial\Omega)]^2 \\
\|\varphibf\|_{ [\Ch]^2}= 1 }}| (\taubf\cdot \varphibf, \psi)_{\Ch}|
&\ge  \left(\frac{ \taubf\cdot(\taubf\psi)}{\|\taubf\psi\|_{[\Ch]^2}}, \psi\right)_{\Ch}
%& = \frac{\|\psi\|_{\Ch}^2 }{\|\taubf\psi\|_{[\Ch]^2}}
\geq \frac{\|\psi\|_{\Ch}}{c_{\Omega} \|\taubf\|_{C^k}}\,.
\end{align*}
\end{proof}

\begin{lemma}\label{lem:rxwelldefined}
Let $\Vx\in \partial \Omega$. The saddle point problem \cref{eq:saddlepoint} admits a unique solution, which satisfies
\begin{equation}\label{eq:apriori_estimates}
\|\Vr_\Vx\|_{[\Ch]^2} \le
\frac{1}{ \Vert \nubf \Vert_{[\Ch]^2}}\,,
\quad \text{and} \quad
\|p_\Vx\|_{\Ch} \le \frac{2 c_{\Omega} \|\taubf\|_{C^k} 
}{ \Vert \nubf \Vert_{[\Ch]^2}}\,.
\end{equation}
\end{lemma}
\begin{proof}
Since $B$ is surjective and $(\cdot,\cdot)_{[\Ch]^2}$ is an inner product,
existence and uniqueness follow from the
classical result \cite[Thm 4.2.1, p.224]{BoBrFo13}. The continuity estimates \cref{eq:apriori_estimates} follow
from \cite[Thm 4.2.3, p.228]{BoBrFo13}, \cref{eq:coercivity_Bt}, and the estimate
\begin{equation}
\sup_{\substack{\varphibf \in [\Ch]^2,\\  \|\varphibf\|_{ [\Ch]^2}= 1}}|\nubf(\Vx)\cdot \varphibf(\Vx)|
\le \frac{|\nubf(\Vx)\cdot \nubf(\Vx)|}{ \Vert \nubf \Vert_{[\Ch]^2}} \leq1/ \Vert \nubf \Vert_{[\Ch]^2}\,.
\end{equation}
{Note that $1\in\Ch(\partial\Omega)$ by \cref{ass:one}, and thus $\nubf \in [\Ch(\partial\Omega)]^2$.}
\end{proof}

The following remark provides an alternative definition of the function $\Vr_\Vx$.
\begin{remark}
Let $\Vx\in \partial \Omega$. The solution $\Vr_\Vx\in [\Ch(\partial\Omega)]^2 $ of \cref{eq:saddlepoint}
is the unique minimiser of 
\ben\label{eq:rx_minimiser}
\min_{\substack{ \varphibf\in [\Ch(\partial\Omega)]^2, \\ \varphibf \cdot \taubf=0 \text{ on } \partial \Omega }}
\frac{1}{2} \|\varphibf\|_{[\Ch]^2}^2 - \nubf(\Vx)\cdot \varphibf(\Vx).
\een
\end{remark}
\begin{proof}
By differentiating the Lagrangian functional
\begin{equation}
\Cl:[\Ch(\partial\Omega)]^2\times \Ch(\partial\Omega)\to \VR\,,
\quad (\varphibf,\psi)\mapsto\frac{1}{2} \|\varphibf\|_{[\Ch]^2}^2 - \nubf(\Vx)\cdot \varphibf(\Vx) +
(\varphibf \cdot \taubf,\psi)_\Ch\,,
\end{equation}
it is easy to see that equations \cref{eq:saddlepoint} are the first order optimality conditions of \cref{eq:rx_minimiser}. This is the so-called Lagrange multiplier rule. 
Since $\varphibf \mapsto 1/2\|\varphibf\|_{[\Ch]^2}^2 - \nubf(\Vx)\cdot \varphibf(\Vx)$ is a convex and coercive functional,
 $\Vr_\Vx$ is a minimiser of \cref{eq:rx_minimiser}, where $(\Vr_\Vx,p_\Vx)$ is a solution of the 
 saddle point equation \cref{eq:saddlepoint}.
\end{proof}

In the next theorem, we analyse in details the properties of $\Vr_\Vx$ and show
that, in general, $\Vr_\Vx$ is neither $\nubf(\cdot)\ksf(\Vx,\cdot)$ nor $\nubf(\cdot)$.
To facilitate the interpretation, in \cref{fig:plotrx} we include some plots of
the function $\Vr_\Vx$.
\begin{theorem}\label{thm:propsofrx}
Let $\Vx\in \partial \Omega$. The solution $(\Vr_\Vx, p_\Vx)$ of \cref{eq:saddlepoint} has the following properties.
\begin{enumerate}[$(a)$]

\item \label{it:normal}
$\Vr_\Vx\in[\Ch(\partial \Omega)]^2_{\nubf}$.
\item \label{it:upperboundnorm}
$\|\Vr_\Vx\|_{[\Ch]^2}\le 1$.

\item \label{it:nonzero}
$\|\Vr_\Vx\|_{[\Ch]^2}>0$, and the surface measure of 
$\supp(\Vr_\Vx)$ is strictly greater than zero.

\item \label{it:norm1implies}
The equality $\|\Vr_\Vx\|_{[\Ch]^2} = 1$ holds if and only if $\partial \Omega\cap \supp(\Vr_\Vx)$
is a straight segment.
In particular, $(\Vr_\Vx, p_\Vx)  = (\ksf(\Vx,\cdot) \nubf(\Vx), 0)$.

\item \label{it:prodnuH}
$(\Vr_\Vx,\nubf)_{[\Ch]^2}=1$ and $\|\nubf\|_{[\Ch]^2}\|\Vr_\Vx\|_{[\Ch]^2}= 1$.  In particular,  if $\|\nubf\|_{[\Ch]^2}=1$, then $\partial \Omega$
is a straight segment.

\item \label{it:whatrxisnot}
If $\partial \Omega\cap \supp(\Vr_\Vx)$
is not a straight segment, then
\begin{equation}
\Vr_\Vx\neq\nubf(\cdot) \ksf(\Vx,\cdot) \quad \text{and} \quad \Vr_\Vx\neq\nubf\,.
\end{equation}

\item \label{it:symm}
The basis functions $\Vr_\Vx,\Vr_\Vy$ associated with the points $\Vx, \Vy\in \partial \Omega$
satisfy
\begin{equation}\label{eq:basis_rx_2}
 \Vr_{\Vx}(\Vy)\cdot \nubf(\Vy) = \Vr_{\Vy}(\Vx)\cdot \nubf(\Vx).
\end{equation}
Moreover, for each $\Vx\in \partial \Omega$ there is a neighborhood $U_\Vx$ of $\Vx$
such that $\Vr_\Vx(\Vy)\cdot \nubf(\Vy)>0$ for all $\Vy\in U_\Vx$.
\end{enumerate}
\end{theorem}

\begin{proof}
\cref{it:normal} follows from \cref{eq:2RKHS2}, because
$\Vr_\Vx (\Vy) \cdot \taubf(\Vy) = (\Vr_\Vx \cdot \taubf,\ksf(\Vy, \cdot))_\Ch = 0$ for every
$\Vy\in\partial\Omega$.

To show \cref{it:upperboundnorm}, note that choosing
$\varphibf = \Vr_\Vx$ in \cref{eq:1RKHS2} gives
\begin{equation}\label{eq:rx_square}
\|\Vr_\Vx\|_{[\Ch]^2}^2 = \Vr_\Vx(\Vx)\cdot \nubf(\Vx) \le |\Vr_\Vx(\Vx)| =
\sqrt{|\Vr_\Vx(\Vx)\cdot \Ve_1|^2 + |\Vr_\Vx(\Vx)\cdot \Ve_2|^2}.
\end{equation}
The reproducing kernel property of $\ksf$ and \cref{eq:nomalizedk} imply
\begin{equation}\label{eq:normkisone}
\|\ksf(\Vx,\cdot)\|_{\Ch}^2 = (\ksf(\Vx,\cdot), \ksf(\Vx,\cdot))_{\Ch} = \ksf(\Vx,\Vx) = 1\,.
\end{equation}
Therefore, by the reproducing kernel property of $\ksf$ and Cauchy-Schwarz inequality,
\begin{equation}\label{eq:rxCSineq}
|\Vr_\Vx(\Vx)\cdot \Ve_i| = |(\ksf(\Vx,\cdot),\Vr_\Vx\cdot \Ve_i)_{\Ch}| \le 
\|\ksf(\Vx,\cdot)\|_{\Ch} \|\Vr_\Vx\cdot \Ve_i\|_{\Ch} =  \|\Vr_\Vx\cdot \Ve_i\|_{\Ch}\,,
\end{equation}
and thus,
\begin{equation*} 
|\Vr_\Vx(\Vx)|^2 = |\Vr_\Vx(\Vx)\cdot \Ve_1|^2 + |\Vr_\Vx(\Vx)\cdot \Ve_2|^2
\leq  \|\Vr_\Vx\cdot \Ve_1\|_{\Ch}^2 +  \|\Vr_\Vx\cdot \Ve_2\|_{\Ch}^2 
=  \|\Vr_\Vx\|_{[\Ch]^2}^2.
\end{equation*}
In light of \cref{eq:rx_square}, this yields 
\begin{equation}\label{eq:in_light}
\begin{split}
\|\Vr_\Vx\|_{[\Ch]^2}^2   \le |\Vr_\Vx(\Vx)|
\le \|\Vr_\Vx\|_{[\Ch]^2},
\end{split}
\end{equation}
which, in turn, implies $\|\Vr_\Vx\|_{[\Ch]^2} \le 1$.

To show \cref{it:nonzero}, suppose that $\Vr_\Vx =0$. Then, $p_\Vx = 0$, because 
\cref{eq:1RKHS2} implies that for any $\Vy\in\partial\Omega$
\begin{align*}
p_\Vx(\Vy) &= (\ksf(\Vy,\cdot), p_\Vx )_{\mathcal H} 
= (\ksf(\Vy,\cdot)\taubf(\cdot) \cdot\taubf(\cdot), p_\Vx )_{\mathcal H}\\
&= \nubf(\Vx)\cdot \ksf(\Vy,\Vx)\taubf(\Vx) - (\Vr_\Vx, \ksf(\Vy,\cdot)\taubf(\cdot))_{[\Ch]^2} = 0\,.
\end{align*}
However, $(\Vr_\Vx,p_\Vx) = 0$ does not satisfy \cref{eq:1RKHS2} when the test function
$\varphibf$ is
$\varphibf=\ksf(\Vx,\cdot)\nubf(\cdot)\in[\mathcal H(\partial\Omega)]^2$.
Finally, $\Vr_\Vx\neq \mathbf{0}$ implies that
there is at least a point $\Vy\in \partial \Omega$ such that $\Vr_\Vx(\Vy)\ne 0$.
Since $\Vr_\Vx$ is continuous on $\partial \Omega$, 
 $\supp(\Vr_\Vx)$ contains an
nonempty neighborhood $U_\Vy$ of $\Vy$.  

To show \cref{it:norm1implies}, let us first assume that $\|\Vr_\Vx\|_{[\Ch]^2}=1$.
Then, equation \cref{eq:in_light} implies $|\Vr_\Vx(\Vx)|=\|\Vr_\Vx\|_{[\Ch]^2}$, and since
\begin{equation*}
|\Vr_\Vx(\Vx)\cdot \Ve_1|^2 + |\Vr_\Vx(\Vx)\cdot \Ve_2|^2
=|\Vr_\Vx(\Vx)|^2=\|\Vr_\Vx\|_{[\Ch]^2}^2= \|\Vr_\Vx\cdot \Ve_1\|_{\Ch}^2 + \|\Vr_\Vx\cdot \Ve_2\|_{\Ch}^2\,,
\end{equation*}
equation \cref{eq:rxCSineq} becomes an equality
(because $|\Vr_\Vx(\Vx)\cdot \Ve_i|  \le  \|\Vr_\Vx\cdot \Ve_i\|_{\Ch}$).
In particular, this implies that
\begin{equation*}
|(\ksf(\Vx,\cdot),\Vr_\Vx\cdot \Ve_i)_{\Ch}|  
=\|\ksf(\Vx,\cdot)\|_{\Ch}  \|\Vr_\Vx\cdot \Ve_i\|_{\Ch}\,.
\end{equation*}
For nonzero vectors, Cauchy-Schwarz inequality becomes an equality 
if and only if the two vectors are linearly dependent.
Since $\ksf(\Vx,\cdot)$ and $\Vr_\Vx$ are both nonzero, there is a nonzero vector
$\alphabf\in\VR^2$ such that $\Vr_\Vx(\cdot) = \ksf(\Vx,\cdot)\alphabf$ (in fact,
$\vert \alphabf \vert = 1$ because $|\Vr_\Vx(\Vx)|=1 = \ksf(\Vx, \Vx)$). This, combined
with $\Vr_\Vx\in[\Ch(\partial \Omega)]^2_{\nubf}$ (see \cref{it:normal}), implies that
$\nubf$ restricted to $\partial \Omega\cap \supp(\ksf(\Vx,\cdot))$ is either $\alphabf$
or $-\alphabf$, and thus constant.

On the other hand, note that
\begin{equation*}(\ksf(\Vx,\cdot)\nubf(\Vx), \varphibf)_{[\Ch]^2} = 
\nubf(\Vx)\cdot\Ve_1(\ksf(\Vx,\cdot), \varphibf\cdot\Ve_1)_{\Ch}
+ \nubf(\Vx)\cdot\Ve_2(\ksf(\Vx,\cdot), \varphibf\cdot\Ve_2)_{\Ch}
= \nubf(\Vx)\cdot \varphibf(\Vx)\,,\end{equation*}
which implies that $(\Vr_\Vx, p_\Vx)  = (\ksf(\Vx,\cdot) \nubf(\Vx), 0)$ satisfies \cref{eq:1RKHS2}.
If $\partial \Omega\cap \supp(\Vr_\Vx)$ is a straight segments,
then $\nubf(\Vx)\cdot\taubf(\Vy)=0$ for all $\Vy\in\partial \Omega\cap \supp(\ksf(\Vx,\cdot))$.
Since $\ksf(\Vx,\Vy)=0$ for all $\Vy\not\in\supp(\ksf(\Vx,\cdot))$,
\begin{equation*}(\ksf(\Vx,\cdot)\nubf(\Vx)\cdot\taubf(\cdot), \ksf(\Vy,\cdot))_\Ch = \ksf(\Vx,\Vy)\nubf(\Vx)\cdot\taubf(\Vy) = 0
\quad \text{for every }\Vy\in\partial\Omega\,,\end{equation*}
and Equation \cref{eq:2RKHS2} is also satisfied.
Finally, $\Vert \ksf(\Vx,\cdot)\nubf(\Vx)\Vert^2_{[\Ch]^2} = 
\vert\nubf(\Vx)\vert^2\Vert\ksf(\Vx,\cdot)\Vert^2_{\Ch} = 1$ by \cref{eq:normkisone}.

To show \cref{it:prodnuH}, recall that $1\in[\Ch(\partial\Omega)]^2$ by \cref{ass:one}.
Therefore, $\nubf \in [\Ch(\partial\Omega)]^2$, and replacing by $\varphibf = \nubf$ in
\cref{eq:1RKHS2} shows that $(\Vr_\Vx,\nubf)_{[\Ch]^2}=1$.
So, by Cauchy-Schwarz inequality and \cref{eq:apriori_estimates},
\begin{equation}
1 = (\Vr_\Vx,\nubf)_{[\Ch]^2} \leq \|\nubf\|_{[\Ch]^2}\|\Vr_\Vx\|_{[\Ch]^2} \leq 1\,.
\end{equation}
Moreover, if $\|\nubf\|_{[\Ch]^2}=1$, then $\|\Vr_\Vx\|_{[\Ch]^2}=1$ for every $\Vx\in\partial\Omega$.
By \cref{it:norm1implies}, it follows that $\partial \Omega\cap \supp(\Vr_\Vx)$ is a straight segment.
Since this is the case for every $\Vx\in\partial\Omega$, it follows that $\partial \Omega$ is a straight segment itself.

To show \cref{it:whatrxisnot}, let $\Vy\in \partial \Omega$.
Plugging $\nubf(\cdot) \ksf(\Vy,\cdot)$ into \cref{eq:1RKHS2} shows that
\begin{equation}
(\Vr_\Vx,\nubf(\cdot) \ksf(\Vy,\cdot))_{[\Ch]^2} = \ksf(\Vx,\Vy).
\end{equation}
In particular, $(\Vr_\Vx,\nubf(\cdot) \ksf(\Vx,\cdot))_{[\Ch]^2} = \ksf(\Vx,\Vx) = 1$.
Therefore,
\begin{align}
0\leq\Vert \Vr_\Vx - \nubf(\cdot) \ksf(\Vx,\cdot)\Vert_{[\Ch]^2}^2&= 
\Vert \Vr_\Vx \Vert_{[\Ch]^2}^2 + \Vert \nubf(\cdot) \ksf(\Vx,\cdot)\Vert_{[\Ch]^2}^2 - 2\,\\
&=(\Vert \Vr_\Vx \Vert_{[\Ch]^2}^2 -1)+ (\Vert \nubf(\cdot) \ksf(\Vx,\cdot)\Vert_{[\Ch]^2}^2 -1)\,.
\end{align}
Since \cref{it:norm1implies} implies that
$\Vert \Vr_\Vx \Vert_{[\Ch]^2} < 1$, we conclude that
$\Vert \nubf(\cdot) \ksf(\Vx,\cdot)\Vert_{[\Ch]^2} > 1$, and $\Vr_\Vx\neq\nubf(\cdot) \ksf(\Vx,\cdot)$.
Finally, \cref{it:prodnuH} states that $(\Vr_\Vx,\nubf)_{[\Ch]^2}=1$. Therefore,
\begin{equation}
0\leq\Vert \Vr_\Vx - \nubf\Vert_{[\Ch]^2}^2
=(\Vert \Vr_\Vx \Vert_{[\Ch]^2}^2 -1)+ (\Vert \nubf\Vert_{[\Ch]^2}^2 -1)\,,
\end{equation}
and $\Vr_\Vx\neq\nubf$ because $\Vert \nubf \Vert_{[\Ch]^2} > 1$. 

To show \cref{it:symm},
note that $(\Vr_{\Vx}, \Vr_\Vy)_{[\Ch]^2} = (\Vr_{\Vy}, \Vr_\Vx)_{[\Ch]^2}$.
Therefore, by \cref{eq:1RKHS2},
\begin{equation}
\nubf(\Vx)\cdot\Vr_{\Vy}(\Vx) = (\Vr_{\Vx}, \Vr_\Vy)_{[\Ch]^2} = (\Vr_{\Vy}, \Vr_\Vx)_{[\Ch]^2} = \nubf(\Vy)\cdot\Vr_{\Vx}(\Vy)\,.
\end{equation}
For $\Vx = \Vy$, this implies $\Vr_\Vx(\Vx)\cdot \nubf(\Vx) =\|\Vr_\Vx\|_{[\Ch]^2}^2>0$.
Since $\Vy\mapsto \Vr_\Vx(\Vy)\cdot \nubf(\Vy)$ is continuous, there is a neighborhood $U_\Vx$ of $\Vx$
such $\Vr_\Vx(\Vy)\cdot \nubf(\Vy) >0$ for all $\Vy\in U_\Vx$. 
\end{proof}

\begin{remark}
The vector field $\Vr_\Vx$ depends on the metric induced by the kernel $\ksf$. For instance,
varying the parameters of the Wendland-kernel \cref{eq:k4} results in different $\Vr_\Vx$s; see \cref{fig:plotrx}.
In particular, we observe that the decay of $\Vr_\Vx$ away from $\Vx$ is influenced by the support of the
kernel $k_p^\sigma$.
\end{remark}

{
\begin{remark}
  To extend \cref{def:rx} to the three dimensional (3D) case, notice that any vector field $\varphibf\in[\Ch(\partial\Omega)]^3$ can be rewritten as $\varphibf = (\varphibf\cdot\nubf)\nubf
+ \VP_{\taubf}(\varphibf)$, where $\VP_{\taubf}(\varphibf)\coloneqq
\varphibf - (\varphibf\cdot\nubf)\nubf$. The operator $\VP_{\taubf}:[\Ch(\partial\Omega)]^3
\to[\Ch(\partial\Omega)]^3$ is a linear continuous projection with closed range $[\Ch(\partial\Omega)]^3_{\taubf}$. Therefore, to construct weakly-normal
vector fields in 3D, one can replace \cref{eq:saddlepoint} with the saddle-point problem
\begin{subequations}
\begin{align}\label{eq:rx3d}
(\Vr_\Vx,\varphibf)_{[\Ch]^3}  +(\VP_{\taubf}(\varphibf), \Vp_\Vx )_{[\Ch]^3}
&= \nubf(\Vx)\cdot \varphibf(\Vx)&& \text{ for all } \varphibf \in [\mathcal H(\partial\Omega)]^3\,,\\
\;(\VP_{\taubf}(\Vr_\Vx), \psibf)_{[\Ch]^3}
& =0 &&\text{ for all } \psibf \in[\Ch(\partial\Omega)]^3\,.
\end{align}
\end{subequations}
Similarly to \cref{lem:B_surjective}, one considers the operator
$B:[\mathcal H(\partial\Omega)]^3 \rightarrow ([\Ch(\partial\Omega)]^3)^*$ defined by
$\varphibf \mapsto (\VP_{\taubf}(\varphibf), \cdot)_{[\Ch]^3}$.
Differently from the 2D case, this operator $B$ is not surjective.
However, it has closed range. Therefore, following \cite[Sect. 4.2.4, p.230]{BoBrFo13},
it is straightforward to show that \cref{eq:rx3d} admits a solution. The function
$\Vr_\Vx$ (which is our main target)
is uniquely determined and satisfies a stability condition similar to
\cref{eq:apriori_estimates}, whereas the multiplier $\Vp_\Vx$ is uniquely
defined only as an element of the quotient space $[\mathcal H(\partial\Omega)]^3/\ker{B^*}$.
Finally, it is possible to derive a counterpart of \cref{thm:propsofrx} by repeating its proof
and replacing products of the form $\Vf\cdot \taubf$ with $\VP_{\taubf}(\Vf)$ in the
proof and in the definition \cref{eq:Hnu}.
\end{remark}
}

\begin{figure}[htb!]
\includegraphics[width=0.45\linewidth]{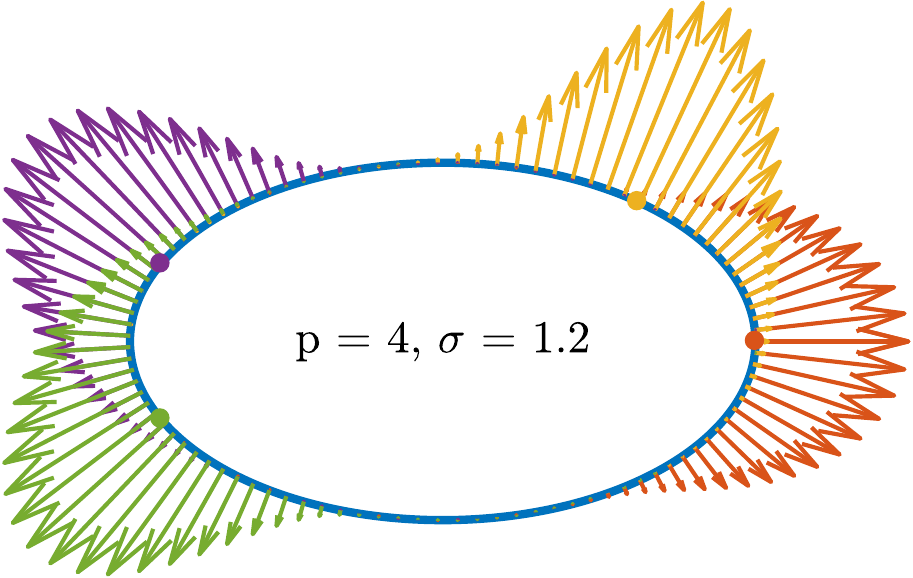}
\hfill
\includegraphics[width=0.45\linewidth]{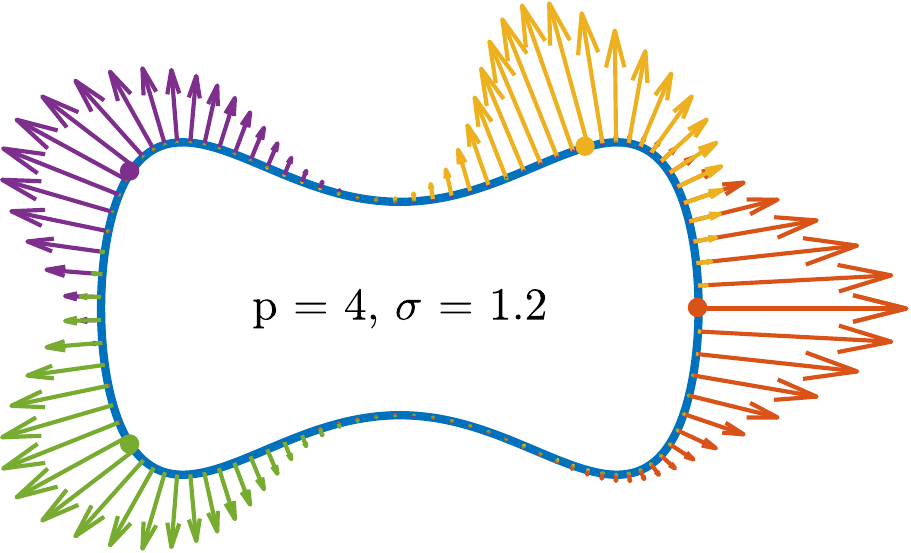}\\
\includegraphics[width=0.45\linewidth]{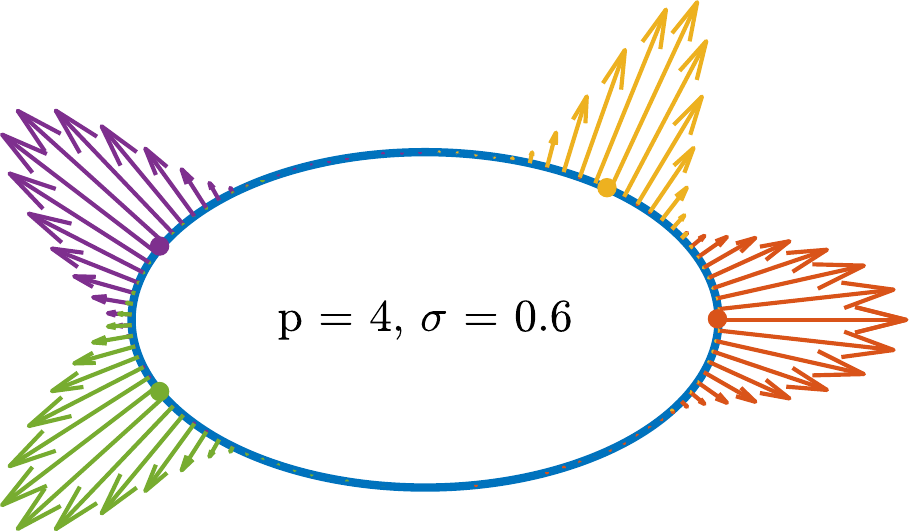}
\hfill
\includegraphics[width=0.45\linewidth]{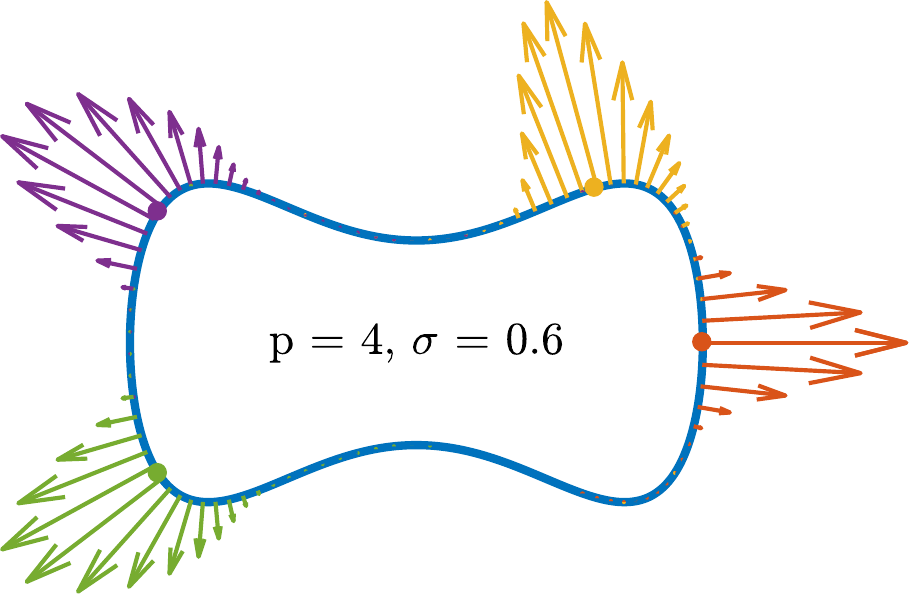}\\
\includegraphics[width=0.45\linewidth]{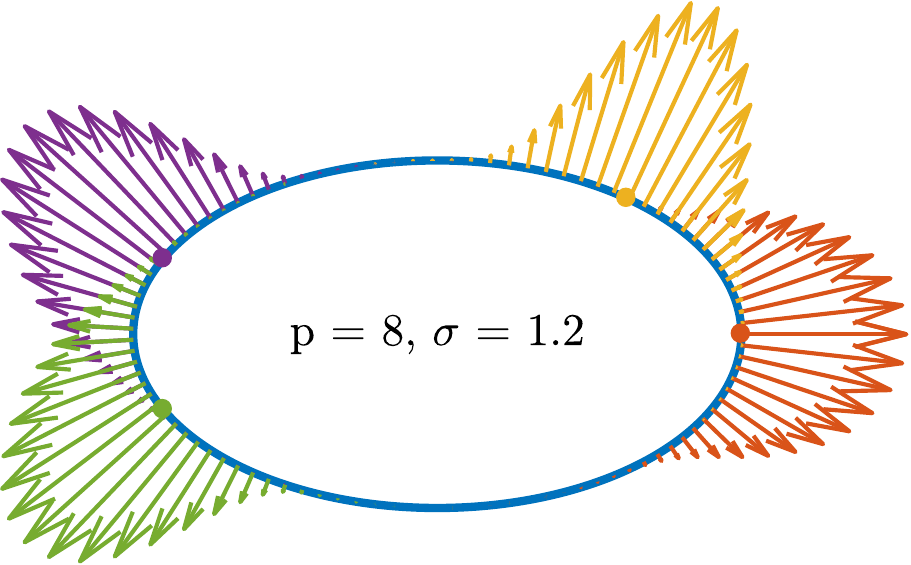}
\hfill
\includegraphics[width=0.45\linewidth]{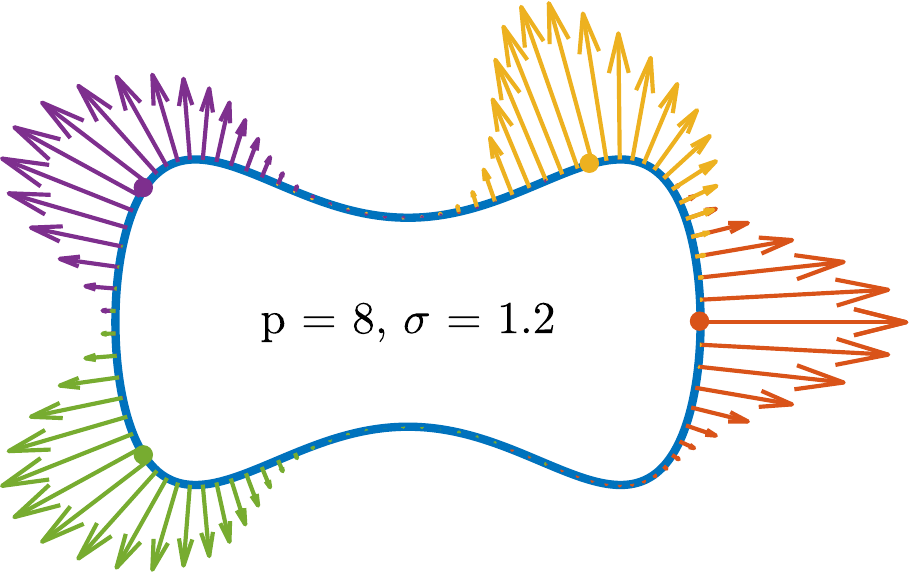}
\caption{Some weakly-normal basis functions based on compactly supported
Wendland kernels \cref{eq:k4} and  \cref{eq:k8} when $\partial\Omega$ is
the ellipse $\{ \gamma(\phi) = (1.4\cos(\phi), 0.8\sin(\phi)): \phi\in[0, 2\pi)\}$ (\emph{left})
and the nonconvex domain
$\{ \gamma(\phi) =(1.4\cos(\phi), 0.8\sin(\phi) + 0.3\sin(3\phi)): \phi\in[0, 2\pi)\}$
(\emph{right}). The dots denote the points to which the weakly-normal basis functions
are associated with.
We observe that different choices of kernel parameters result in different weakly-normal
basis functions.}
\label{fig:plotrx}
\end{figure}

\subsection{Density of weakly-normal functions}
In this section we study the approximation properties of the functions
$\Vr_\Vx$s with respect to $[\Ch(\partial \Omega)]^2_{\nubf}$.

Let $N\ge 1$ be an integer and let $\Cx_N \coloneqq \{\Vx_1,\ldots, \Vx_N \}\subset \partial \Omega$
be a collection of pairwise distinct points.
\begin{lemma}\label{lem:linearindepenent}
The functions $\Vr_{\Vx_1},\ldots, \Vr_{\Vx_N}$ associated with the points in $\Cx_N$
are linearly independent.
\end{lemma}
\begin{proof}
Let $\gamma_1,\ldots, \gamma_N\in \VR$ be such that
$\Vr\coloneqq \sum_{i=1}^N \gamma_i \Vr_{\Vx_i} = 0$.
Setting $p \coloneqq \sum_{i=1}^N \gamma_i p_{\Vx_i}$
and summing \cref{eq:1RKHS2} over $i=1,\ldots, N$ yields
\begin{equation}\label{eq:p_var}
(\varphibf\cdot \taubf,p)_{\mathcal H} = \sum_{i=1}^N \gamma_i\nubf(\Vx_i)\cdot \varphibf(\Vx_i) = \bigg(\sum_{i=1}^N \gamma_i \ksf(\Vx_i,\cdot),\, \nubf\cdot \varphibf\bigg)_{\mathcal H}
\quad \text{for all }\varphibf\in [\Ch(\partial \Omega)]^2\,.
\end{equation} 
Choosing
$\varphibf = (\sum_{i=1}^N \gamma_i k(\Vx_i,\cdot))\nubf \in[\Ch(\partial \Omega)]_{\nubf}^2$
(cf. \cref{ass:one}), yields
\begin{equation}
\Vert\sum_{i=1}^N \gamma_i \ksf(\Vx_i,\cdot)\Vert_{\mathcal H}^2 = 0\,,
\quad\text{and thus,}\quad \sum_{i=1}^N \gamma_i \ksf(\Vx_i,\Vy) =0
\quad \text{ for all } \Vy\in \VR^2.
\end{equation}
In particular, this implies that
\begin{equation}
\sum_{i=1}^N \gamma_i \ksf(\Vx_i,\Vx_j) =0
\quad \text{ for all } \Vx_j \in \Cx_N,
\end{equation}
and thus, $\gamma_1=\cdots = \gamma_N=0$, because $\ksf$ is positive-definite on $\VR^2$.
Therefore, the functions
$\Vr_{\Vx_1},\ldots, \Vr_{\Vx_N}$	are linearly independent.
\end{proof}

\Cref{lem:linearindepenent} allows the definition of the interpolation operator onto
$\Cr_N\coloneqq\mathrm{span}\{\Vr_{\Vx}:\Vx\in\Cx_N\}$. The following lemma and remark
clarify the approximation properties of $\Cr_N$ on $[\Ch(\partial \Omega)]^2$.

\color{black}
\begin{lemma}\label{lem:best_approximation}
The interpolation operator $[\Ch(\partial \Omega)]^2_{\nubf}\ni\Vr\mapsto \Vr_N\in\Cr_N$,
defined by 
\begin{equation}\label{eq:interpolation}
\Vr(\Vx) = \Vr_N(\Vx) \quad \text{for every } \Vx\in\Cx_N
\end{equation}
is an $[\Ch]^2$-orthogonal projection.
\end{lemma}
\begin{proof}
First of all, recall that $\Vr_{\Vx}\in[\Ch(\partial \Omega)]^2_{\nubf}$ for every $\Vx\in\Cx_N$.
Therefore, $\Vr_N \in [\Ch(\partial \Omega)]^2_{\nubf}$. For $\Vx\in \Cx_N$ fixed, using $\Vr$ and $\Vr_N$ as a test functions in \cref{eq:1RKHS2}
gives
\begin{equation}
(\Vr_\Vx,\Vr)_{[\Ch]^2} = \nubf(\Vx)\cdot \Vr(\Vx) \stackrel{\cref{eq:interpolation}}{=} 
\nubf(\Vx)\cdot\Vr_N(\Vx) = (\Vr_\Vx,\Vr_N)_{[\Ch]^2}\,.
\end{equation}
Therefore,
\begin{equation}\label{eq:orthogonality}
(\Vr_\Vx, \Vr-\Vr_N)_{[\Ch]^2}=0 \quad \text{ for all } \Vx\in \Cx_N.
\end{equation}
\end{proof}
\begin{remark}\label{rmk:best_approximation}
\cref{lem:best_approximation} implies that $\Vr_N$ is the best approximation
of $\Vr$ in $\Cr_N$ with respect to the $[\Ch]^2$-norm,
and that $\Vert \Vr_N \Vert_{[\Ch]^2} \leq \Vert \Vr\Vert_{[\Ch]^2}$.
\end{remark}

We conclude the section by showing that any function in $[\Ch(\partial \Omega)]^2_{\nubf}$
can be approximated arbitrarily well using sufficiently many $\Vr_\Vx$s.
\begin{theorem}\label{thm:rxaredensesubset}
Let $\{\Cx_N\}_{N\in\VN}$ be a nested sequence of
finite subsets (that contain pairwise distinct points) of $\partial \Omega$
such that the union $\cup_{N\in\VN}\Cx_N$ is dense in $\partial \Omega$.
The set	$\cup_{N\ge 1}\Cr_N$ is dense in
$[\Ch(\partial \Omega)]^2_{\nubf}$.
\end{theorem}
\begin{proof}

For an arbitrary $\Vr\in [\Ch(\partial \Omega)]^2_{\nubf}$, denote
by $\Vr_N$ its interpolant on $\Cr_N$
(see \cref{lem:best_approximation}).
The sequence $\{\Vr_N\}_{N\in\VN}$ is bounded in $[\Ch(\partial \Omega)]^2$ because
$\Vert \Vr_N \Vert_{[\Ch]^2} \leq \Vert \Vr\Vert_{[\Ch]^2}$ for every $N$.
Therefore, there is a subsequence $\{\Vr_{N_i}\}_{i\in\VN}$
that converges  weakly to a function $\Vf\in [\Ch(\partial \Omega)]^2_{\nubf}$.
By Mazur's Lemma \cite[p.61]{TemanEkeland99}, there is a sequence $\{\hat\Vr_{n}\}_{n\in\VN}$
of convex combinations of $\{\Vr_{n_i}\}_{i\in\VN}$ that converges
strongly in $[\Ch(\partial \Omega)]^2$ to $\Vf$.
Note that there is a function $\Cn:\VN \to \VN$ such that $\Cn(n)\ge n$ and  $\hat\Vr_{n} \in  
\Cr_{\Cn(n)}$ for every $n\in\VN$.

We recall that in RKHS strong convergence implies pointwise convergence. Therefore,
for every $\Vz\in\partial \Omega$,
\begin{equation}\label{eq:conv_pointwise}
\hat \Vr_{n}(\Vz) \to \Vf(\Vz) \quad \text{ as } n\to \infty\,.
\end{equation}
Additionally,
\begin{equation}
\hat\Vr_{n}(\Vz) = \Vr(\Vz) \quad \text{for every } \Vz\in \Cx_{\Cn(n)}\,,
\end{equation}
because $\hat\Vr_{n}$ is a convex combination of interpolants. Therefore, 
$\Vf = \Vr$ on $\cup_{n\in\VN}\Cx_{\Cn(n)}$ and in turn, since $\cup_{n\in\VN}\Cx_{\Cn(n)}$ is dense in $\partial\Omega$ and $\Vf$ and $\Vr$ are continuous, we conclude $\Vf = \Vr$ on $\partial \Omega$. 
It follows that  $\hat \Vr_{n}$ converges strongly to $\Vr$.
\end{proof}
%%%%%%%%%%%%%%%%%%%%%%%%%%%%%%%%%%%%%%%%%%%%%%%%%%%%%%%%
%%%%%%%%%%%%%%%%%%%%%%%%%%%%%%%%%%%%%%%%%%%%%%%%%%%%%%%%
\subsection{Approximation of weakly-normal basis functions}
The functions $\Vr_\Vx$s introduced by \cref{def:rx} live in the infinite dimensional space
$[\Ch(\partial \Omega)]^2$. In this section, we explain how to approximate them using
finitely many collocation points and analyse the numerical error of this approximation.

In this section, the sets $\mathcal X = \{\Vx_1,\ldots, \Vx_N\}$ and $\mathcal Y = \{\Vy_1,\ldots, \Vy_M\} $, $M\le N$,  
denote two subset of $\partial \Omega$ (containing pairwise distinct points) with $\Cy \subset\Cx$.
These sets are used to define the finite dimensional spaces
\begin{equation}\label{eq:V_h}
V_h(\Cx) = \mathrm{span}\{\ksf(\Vx, \cdot)|_{\partial \Omega}: \Vx\in \Cx\}
\quad \text{and} \quad V_h(\Cy) = \mathrm{span}\{\ksf(\Vy, \cdot)|_{\partial \Omega}: \Vy\in \Cy\}.
\end{equation}
Finally, we denote by $\Ci_{V_h}:\Ch(\partial \Omega) \to V_h(\Cx)$ and
$\Ci_{[V_h]^2}:[\Ch(\partial \Omega)]^2 \to [V_h(\Cx)]^2$ the standard pointwise interpolation operators.

\begin{definition}\label{def:rxh}
Let $\Vx\in \partial \Omega$. The finite dimensional approximation
$(\Vr_\Vx^h,p_\Vx^h)\in[V_h(\Cx)]^2\times [V_h(\Cy)]$ of $(\Vr_\Vx, p_\Vx)$ is characterised by
\begin{subequations}
\label{eq:saddlepointdiscr}
\begin{align}
(\Vr_\Vx^h,\varphibf)_{[\Ch]^2}  +(\varphibf \cdot \taubf, p_\Vx^h )_{\mathcal  H}
&= \nubf(\Vx)\cdot \varphibf(\Vx)&& \text{ for all } \varphibf \in [V_h(\mathcal X)]^2,\label{eq:1RKHSh2}\\
\;(\Vr_\Vx^h\cdot \taubf, \psi)_{\mathcal H}
& =0 &&\text{ for all } \psi \in V_h(\mathcal Y) \label{eq:2RKHSh2}\,.
\end{align}
\end{subequations}
\end{definition}

\Cref{thm:uniquenessdiscr} shows that \cref{def:rxh} is well defined.
Its proof relies on \cref{lem:infsup},
which is discrete counterpart of \cref{lem:B_surjective}.
To prove this lemma, we use the following properties of the interpolation operators
$\Ci_{V_h}$ and $\Ci_{[V_h]^2}$.

\begin{lemma}
The interpolation operator $\Ci_{V_h}$
is an $\Ch$-orthogonal projection.
\end{lemma}
\begin{proof}
Let $\Vx$ be a point from $\Cx$, and choose a $\psi \in \Ch[\partial\Omega]$.
Since $\Vx$ is an interpolation point, it holds that
\begin{equation}
(\Ci_{V_h}(\psi), \ksf(\Vx, \cdot))_\Ch = \Ci_{V_h}(\psi)(\Vx) = \psi(\Vx) = (\psi, \ksf(\Vx, \cdot))_\Ch\,,
\end{equation}
and thus,
\begin{equation}
(\Ci_{V_h}(\psi) - \psi, \phi)_\Ch = 0 \quad \text{for all } \phi\in V_h(\Cx)\,.
\end{equation}
\end{proof}
\begin{corollary}\label{cor:interpol_operator_continuous1}
If \cref{ass:one} holds true, then
the interpolation operator $\Ci_{V_h}$ satisfies
\begin{equation}\label{eq:estimateInterp}
\Vert \Ci_{V_h}(f\psi) \Vert_\Ch \leq c_\Omega\Vert f \Vert_{C^{k}(\partial\Omega)} \Vert \psi \Vert_\Ch
\end{equation}
for every $f\in C^{k}(\partial\Omega)$ and every $\psi \in \Ch(\partial\Omega)$.
\end{corollary}
\begin{proof}
Choose an $f\in C^{k}(\partial\Omega)$ and a
$\psi \in \Ch(\partial\Omega)$. Since $\Ci_{V_h}$ is an $\Ch$-orthogonal projection, 
$V_h(\Cx)\subset\Ch(\partial\Omega)$, and \cref{eq:module},
it holds that
\begin{equation}
\Vert \Ci_{V_h}(f\psi) \Vert_\Ch \leq\Vert f\psi \Vert_\Ch  \leq c_\Omega\Vert f \Vert_{C^{k}(\partial\Omega)} \Vert \psi \Vert_\Ch\,.
\end{equation}
\end{proof}

\begin{corollary}\label{cor:interpol_operator_continuous2}
If \cref{ass:one} holds true, then the interpolation operator
$\Ci_{[V_h]^2}$ satisfies
\begin{equation}\label{eq:interpol_operator_continuous2}
\Vert \Ci_{[V_h]^2}(\taubf \psi)\Vert_{[\Ch]^2}
\leq c_\Omega \Vert \taubf \Vert_{C^k}
\Vert \psi \Vert_\Ch \quad \text{for all } \psi \in \Ch(\partial\Omega).
\end{equation}
\end{corollary}
\begin{proof}
Note that $\Ci_{[V_h]^2}(\taubf \psi) = \Ci_{V_h}(\taubf\cdot\Ve_1 \psi)\Ve_1 + \Ci_{V_h}(\taubf\cdot\Ve_2 \psi)\Ve_2$.
Therefore, by \cref{cor:interpol_operator_continuous1},
\begin{align*}
\Vert \Ci_{[V_h]^2}(\taubf \psi)\Vert_{[\Ch]^2}^2 &= \Vert \Ci_{V_h}(\taubf\cdot\Ve_1 \psi)\Vert_{\Ch}^2
+ \Vert \Ci_{V_h}(\taubf\cdot\Ve_2 \psi)\Vert_{\Ch}^2\\
&\leq c_\Omega^2 \Vert \taubf\cdot\Ve_1\Vert_{C^k}^2 \Vert \psi \Vert_\Ch^2
+c_\Omega^2 \Vert \taubf\cdot\Ve_2\Vert_{C^k}^2 \Vert \psi \Vert_\Ch^2\\
&= c_\Omega^2\Vert \taubf\Vert_{C^k}^2\Vert \psi \Vert_\Ch^2\,.
\end{align*}
\end{proof}

\begin{lemma}[discrete inf-sup condition]\label{lem:infsup}
The following holds
\ben\label{eq:infsup}
\sup_{\substack{\varphibf \in [V_h(\mathcal X)]^2,\\ \|\varphibf\|_{[\Ch]^2}\le 1} } (\varphibf \cdot \taubf, \psi )_{\mathcal  H}
\geq \frac{\|\psi\|_{\Ch}}{c_\Omega\|\taubf\|_{C^k}}
\quad \text{for all }\psi  \in V_h(\mathcal Y)\,.
\een
\end{lemma}	
\begin{proof}
Let $\psi\in V_h(\Cy)$. By \cref{cor:interpol_operator_continuous2},
\begin{equation}\label{eq:proofinfsup1ststep}
\sup_{\substack{\varphibf \in [V_h(\mathcal X)]^2,\\ \|\varphibf\|_{[\Ch]^2}\le 1} } (\varphibf \cdot \taubf, \psi )_{\mathcal  H} \ge \left(\taubf\cdot\frac{ \Ci_{[V_h]^2}(\taubf\psi)}{\|\Ci_{[V_h]^2}(\taubf\psi)\|_{[\Ch]^2}}, \psi\right)_{\Ch}
\geq \frac{(\taubf\cdot\Ci_{[V_h]^2}(\taubf\psi),\psi)_{\Ch}}{c_\Omega\|\taubf\|_{C^k}\|\psi\|_{\Ch}}\,.
\end{equation}
Let $\Vy\in\Cy$. Then, since $\Cy \subset\Cx$, 
\begin{align*}
(\taubf\cdot\Ci_{[V_h]^2}(\taubf\psi),\ksf(\Vy,\cdot))_{\Ch}
&= \taubf(\Vy)\cdot\Ci_{[V_h]^2}(\taubf\psi)(\Vy)\\
&= \taubf(\Vy)\cdot (\taubf\psi)(\Vy)
= \psi(\Vy)
= (\psi,\ksf(\Vy,\cdot))_{\Ch}\,.
\end{align*}
Therefore, \cref{eq:infsup} follows by replacing $(\taubf\cdot\Ci_{[V_h]^2}(\taubf\psi),\psi)_{\Ch} = (\psi,\psi)_{\Ch} = \Vert\psi\Vert_{\Ch}^2$ into \cref{{eq:proofinfsup1ststep}}.
\end{proof}

\begin{theorem}\label{thm:uniquenessdiscr}
The finite dimensional problem \cref{eq:saddlepointdiscr} has a unique solution, which satisfies
\ben\label{eq:estimate_rh}
\|\Vr^h_\Vx\|_{[\Ch]^2} \le 1 \quad \text{and}\quad 
\Vert p_\Vx^h\Vert_\Ch \leq 2c_\Omega\|\taubf\|_{C^k}\,.% \Vert \taubf\Vert_{C^k}^2. 
\een
Moreover,
\begin{equation}
\label{eq:r_x^hquasiopt}
\|\Vr_\Vx^h - \Vr_\Vx\|_{[\Ch]^2}
\le \Vert \nubf\Vert_{C^k}^2\inf_{\varphibf\in [V_h(\Cx)]^2}\|\Vr_x-\varphibf\|_{[\Ch]^2}
+ \Vert \taubf\Vert_{C^k} \inf_{\psi\in V_h(\Cy)}\|p_\Vx -  \psi\|_{\Ch}\,,
\end{equation}
and
\begin{equation}
\label{eq:p_x^hquasiopt}
\|p_\Vx^h - p_\Vx\|_{\Ch}
\leq c_\Omega \Vert \taubf \Vert_{C^k}
\|\Vr_\Vx^h - \Vr_\Vx\|_{[\Ch]^2}
+  (1 + c_\Omega^2\Vert \taubf \Vert_{C^k}^2)\inf_{\psi\in V_h(\Cy)}\|p_\Vx -  \psi\|_{\Ch}\,.
\end{equation}
\end{theorem}
\begin{proof}
The functions $\Vr_\Vx^h$ and $p_\Vx^h$ can be written as
\begin{equation}
\Vr_\Vx^h =\sum_{i=1}^N \alphabf_{i}\ksf(\Vx_i, \cdot), 
\qquad  p_\Vx^h = \sum_{i=1}^M \beta_{i} \ksf(\Vy_i, \cdot),
\end{equation}
where $\alphabf_{i}=(\alpha_i^1,\alpha_i^2)^\top\in\VR^2$ and $\beta_{i}\in \VR$ are certain unknown coefficients.

The vectors  $\Va \coloneqq (\alpha_{1}^1,\ldots, \alpha_{N}^1, \alpha_{1}^2,\ldots, \alpha_{N}^2)^\top$
and $\Vb \coloneqq (\beta_1,\ldots, \beta_M)^\top$ satisfy
\begin{equation}\label{eq:discretesaddlepoint2}
\begin{pmatrix}
	\VA& \VB^\top\\
	\VB & \bold{0}\;\;  	
\end{pmatrix}
\begin{pmatrix}
	\Va\\ \Vb
\end{pmatrix} = \begin{pmatrix}
\VL\\\bold{0}
\end{pmatrix}\,,
\end{equation}
where \begin{equation}\label{ea:VA}
\VA = \begin{pmatrix}
	\tilde\VA & \bold{0}\\
	\bold{0}    & \tilde\VA
\end{pmatrix}, \quad \VB = \begin{pmatrix}
\VB^1\,\VB^2
\end{pmatrix}, \quad \VL = \begin{pmatrix}
\VL^1\\
\VL_2
\end{pmatrix}\,,
\end{equation}
and (with $\taubf = (\tau^1, \tau^2)^\top$ and  $\nubf = (\nu^1, \nu^2)^\top$)
\begin{align}
\label{eq:Ah2}
(\tilde\VA)_{n,m} &= \ksf(\Vx_n,\Vx_m), &  n, m=1,\ldots, N,\\
\label{eq:Bh2}
(\VB^\ell)_{m,n}   &=  \ksf(\Vx_n,\Vy_m)\tau^\ell(\Vx_n),& m = 1,\ldots, M,\ n=1,\ldots, N, \\\
\label{eq:Fh2}
(\VL^\ell)_n &= \nu^\ell(\Vx) \ksf(\Vx_n,\Vx) ,& \quad n=1,\ldots, N.
\end{align}

The linear system \cref{eq:discretesaddlepoint2} has a unique solution because
$\tilde\VA$ (and thus $\VA$) is positive-definite and $\VB$ has maximal rank  $M$.
The maximal rank of $\VB$ follows directly from \cref{lem:infsup}. For a more
direct proof, recall that $\mathcal Y\subset \mathcal X$.
Therefore, the matrix $\VB$ contains the $M$ column vectors
$\{(\ksf(\Vy_i,\Vy_j)c_i)_{j=1,\ldots, M}:  i=1,\dots, M\}$,
where $c_i = \tau^1(\Vy_i)$ if $\tau^1(\Vy_i)\neq 0$ and
$c_i = \tau^2(\Vy_i)$ otherwise (so that $c_i\neq0$ for every  $i$). These vectors are
linearly independent because the kernel $\ksf$ is positive-definite.

The estimate \cref{eq:estimate_rh} follows by inserting $\varphibf = \Vr_\Vx^h$
into \cref{eq:1RKHSh2} and using Cauchy-Schwarz and $\|\ksf(\Vx,\cdot)\|_{\Ch}=1$
(see \cref{eq:normkisone}). 
 To prove the estimate on the multiplier $p^h_\Vx$, note
that, by \cref{eq:infsup},
\begin{equation}
\Vert p^h_\Vx \Vert \leq c_\Omega\|\taubf\|_{C^k}
\sup_{\substack{\varphibf \in [V_h(\mathcal X)]^2,\\ \|\varphibf\|_{[\Ch]^2}\le 1} }
(\varphibf \cdot \taubf, p^h_\Vx )_{\mathcal  H}\,,
\end{equation}
and that by \cref{eq:1RKHSh2},
for every $\varphibf \in [V_h(\mathcal X)]^2$ with $\|\varphibf\|_{[\Ch]^2}\le 1$,
\begin{align*}
(\varphibf \cdot \taubf, p^h_\Vx)_{\Ch} 
&=\nubf(\Vx)\cdot\varphibf(\Vx)
- (\Vr^h_\Vx, \varphibf)_{[\Ch]^2}\\
&\leq\vert\varphibf(\Vx)\vert + \Vert \Vr^h_\Vx\Vert_{[\Ch]^2} \Vert \varphibf\Vert_{[\Ch]^2}
\leq \Vert \varphibf \Vert_{[\Ch]^2} + 1 \leq 2\,. 
\end{align*}

Finally, quasi-optimality results similar to \cref{eq:r_x^hquasiopt} and
\cref{eq:p_x^hquasiopt} can be derived by applying directly \cite[Prop. 5.2.1, pp 274]{BoBrFo13}.
Here, we give explicit details on the derivation of \cref{eq:r_x^hquasiopt} and \cref{eq:p_x^hquasiopt}.

To show \cref{eq:r_x^hquasiopt}, note that
the differences $\Vr_\Vx^h-\Vr_\Vx$ and  $p_\Vx^h-p_\Vx$ satisfy
\begin{subequations}
	\label{eq:saddlepoint_orthogonal}
	\begin{align}
	(\Vr_\Vx^h-\Vr_\Vx,\varphibf)_{[\Ch]^2}  +(\varphibf \cdot \taubf, p_\Vx^h - p_\Vx )_{\mathcal  H}
	&= 0&& \text{ for all } \varphibf \in[V_h(\mathcal X)]^2,\label{eq:1RKHSh22}\\
	\;((\Vr_\Vx^h-\Vr_\Vx)\cdot \taubf, \psi)_{\mathcal H}
	& =0 &&\text{ for all } \psi \in V_h(\mathcal Y) \label{eq:2RKHSh22}\,.
	\end{align}
\end{subequations}
This implies that
\begin{align}
\nonumber
\|\Vr_\Vx^h - \Vr_\Vx\|_{[\Ch]^2}^2 
&= (\Vr_\Vx^h-\Vr_\Vx,\Vr_\Vx^h-\Vr_\Vx)_{[\Ch]^2}\\
\nonumber
&= (\Vr_\Vx^h-\Vr_\Vx,\Vr_\Vx^h)_{[\Ch]^2} - (\Vr_\Vx^h-\Vr_\Vx,\Vr_\Vx)_{[\Ch]^2}\\
\nonumber
&= -(\Vr_\Vx^h\cdot\taubf,p_\Vx^h - p_\Vx )_{\Ch} - (\Vr_\Vx^h-\Vr_\Vx,\Vr_\Vx)_{[\Ch]^2}\\
\nonumber
&= -(\Vr_\Vx^h\cdot\taubf,p_\Vx^h )_{\Ch} + (\Vr_\Vx^h\cdot\taubf,p_\Vx )_{\Ch}
- (\Vr_\Vx^h-\Vr_\Vx,\Vr_\Vx)_{[\Ch]^2}\\
\nonumber
&= (\Vr_\Vx^h\cdot\taubf,p_\Vx )_{\Ch} - (\Vr_\Vx^h-\Vr_\Vx,\Vr_\Vx)_{[\Ch]^2}\\
\nonumber
&= (\Vr_\Vx^h\cdot\taubf,p_\Vx )_{\Ch} - (\Vr_\Vx^h-\Vr_\Vx,(\Vr_\Vx\cdot\nubf)\cdot\nubf)_{[\Ch]^2}\,.
\end{align}
The latter result combined with \cref{eq:saddlepoint_orthogonal} implies that,
for any $\varphibf \in [V_h(\mathcal X)]^2$ and any $\psi \in V_h(\mathcal Y)$,
\begin{equation}
\|\Vr_\Vx^h - \Vr_\Vx\|_{[\Ch]^2}^2 
= (\Vr_\Vx^h\cdot\taubf-\Vr_\Vx\cdot\taubf,p_\Vx -  \psi)_{\Ch} 
- (\Vr_\Vx^h-\Vr_\Vx,(\Vr_\Vx\cdot\nubf)\cdot\nubf-(\varphibf\cdot\nubf)\nubf)_{[\Ch]^2}\,.
\end{equation}
Then, Cauchy-Schwarz inequality and \cref{eq:module} imply
\begin{align}
\nonumber
\|\Vr_\Vx^h - \Vr_\Vx\|_{[\Ch]^2}^2
&\le \|\Vr_\Vx^h\cdot\taubf-\Vr_\Vx\cdot\taubf\|_{\Ch}\|p_\Vx -  \psi\|_{\Ch}
+ \|\Vr_\Vx^h-\Vr_\Vx\|_{[\Ch]^2} \|((\Vr_x-\varphibf)\cdot \nubf)\nubf\|_{[\Ch]^2}\,,\\
\label{eq:rx-rxhineq}
&\le \|\Vr_\Vx^h-\Vr_\Vx\|_{[\Ch]^2}
( \Vert \taubf\Vert_{C^k} \|p_\Vx -  \psi\|_{\Ch}
+\Vert \nubf\Vert_{C^k}^2\|\Vr_x-\varphibf\|_{[\Ch]^2})\,,
\end{align}
from which \cref{eq:r_x^hquasiopt} follows easily.

To show \cref{eq:p_x^hquasiopt}, note that for any $\psi \in V_h(\mathcal Y)$,
\cref{eq:infsup} implies
\begin{equation}
\Vert p_\Vx^h -\psi \Vert_{\Ch} \leq c_\Omega \Vert \taubf\Vert_{C^k}
\sup_{\substack{\varphibf \in [V_h(\mathcal X)]^2,\\ \|\varphibf\|_{[\Ch]^2}\le 1} }
(\varphibf \cdot \taubf, p_\Vx^h -\psi)_{\mathcal  H}\,.
\end{equation}
For any $\varphibf \in [V_h(\mathcal X)]^2$, \cref{eq:saddlepoint} and \cref{eq:saddlepointdiscr}
imply
\begin{align*}
(\varphibf \cdot \taubf, p_\Vx^h -\psi)_{\Ch}
&= (\varphibf \cdot \taubf, p_\Vx^h)_{\Ch} - (\varphibf \cdot \taubf, \psi)_{\Ch}\\
&= \nubf(\Vx)\cdot\varphibf(\Vx) - (\Vr_\Vx^h, \varphibf)_{[\Ch]^2}  - (\varphibf \cdot \taubf, \psi)_{\Ch}\\
&=(\Vr_\Vx, \varphibf)_{[\Ch]^2} + (\varphibf \cdot \taubf, p_\Vx)_{\Ch}
- (\Vr_\Vx^h, \varphibf)_{[\Ch]^2}  - (\varphibf \cdot \taubf, \psi)_{\Ch}\\
&=(\Vr_\Vx-\Vr_\Vx^h,\varphibf )_{[\Ch]^2} 
+(\varphibf \cdot \taubf, p_\Vx -\psi)_{\Ch}\\
&\leq \Vert \psibf \Vert_{[\Ch]^2}(\Vert \Vr_\Vx-\Vr_\Vx^h \Vert_{[\Ch]^2} 
+ c_\Omega \Vert \taubf \Vert_{C^k}  \Vert\Vert p_\Vx -\psi\Vert_{\Ch})\,,
\end{align*}
where the last inequality follows from Cauchy Schwarz inequality and \cref{eq:module}.
Therefore,
\begin{equation}
\Vert p_\Vx^h -\psi \Vert_{\Ch} \leq
c_\Omega \Vert \taubf\Vert_{C^k}
(\Vert \Vr_\Vx-\Vr_\Vx^h \Vert_{[\Ch]^2} 
+ c_\Omega \Vert \taubf \Vert_{C^k}  \Vert p_\Vx -\psi\Vert_{\Ch})\,,
\end{equation}
and, by the triangle inequality,
\begin{align*}
\Vert p_\Vx - p_\Vx^h \Vert_{\Ch} &\leq \Vert p_\Vx - \psi \Vert_{\Ch} + \Vert p_\Vx^h - \psi \Vert_{\Ch}\\
&\leq c_\Omega \Vert \taubf\Vert_{C^k} \Vert \Vr_\Vx-\Vr_\Vx^h \Vert_{[\Ch]^2} 
+ (1 + c_\Omega^2\Vert \taubf \Vert_{C^k}^2)  \Vert p_\Vx -\psi\Vert_{\Ch})\,.
\end{align*}
\end{proof}

The following corollary follows directly from the quasi-optimality estimates
\cref{eq:r_x^hquasiopt}  and \cref{eq:p_x^hquasiopt} 
%Before proving convergence rates for specific choices of reproducing kernels,
%we show that the approximation $\Vr_\Vx^h$ converges weakly to $\Vr_\Vx$, if the
%sets $\Cx$ and $\Cy$ are chosen appropriately.
\begin{corollary}
Let $\{\Cy_N\}_{N\in\VN}$ be a nested sequence of
finite subsets (that contain pairwise distinct points) of $\partial \Omega$
such that the union $\cup_{N\in\VN}\Cy_N$ is dense in $\partial \Omega$.
%and set $\Cy_N = \Cx_N$ in \cref{eq:saddlepointdiscr}.
%Let $h_N$ be the fill distance associated with $\Cx_N$.
Then, the sequence of functions $(\Vr_\Vx^{h_N},p_\Vx^{h_N})$ converges strongly to
$(\Vr_\Vx,p_\Vx)$.
\end{corollary}

Before studying convergence rates, we mention that in \cite{sturm3}
Sturm proved superlinear convergence of a shape Newton method
in the Micheletti space based on functions
of the form $\ksf(\Vx,\cdot)\nubf(\Vx)$. The next lemma shows
that such functions arise from a particular discretisation of \cref{eq:saddlepoint}. 
\begin{proposition}
If $M=1$ and $\Cy=\{\Vx\}\subset\Cx$, then the solution of \cref{eq:saddlepointdiscr} is
$$
\Vr_\Vx^h(\cdot) = \ksf(\Vx,\cdot)\nubf(\Vx), \quad p^h_\Vx =0\,.
$$
\end{proposition}
\begin{proof}
By \cref{thm:uniquenessdiscr}, we only need to verify that the pair
$(\Vr_\Vx^h, p^h_\Vx)$ satisfies \cref{eq:saddlepointdiscr}.
By the reproducing kernel property,
\[ 
(\Vr_\Vx^h, \varphibf)_{[\Ch]^2} = \nubf(\Vx)\cdot \varphibf(\Vx)
\quad \text{for all } \varphibf \in [\Ch(\partial\Omega)]^2\,,
\]
which, together with $p^h_\Vx =0$, implies \cref{eq:1RKHSh2}.
Finally,
\begin{equation}
(\Vr^h_\Vx\cdot \taubf, \ksf(\Vx, \cdot))_{\mathcal H}
= \ksf(\Vx, \Vx)\taubf(\Vx)\cdot \nubf(\Vx) = 0\,,
\end{equation}
together with  $V_h(\Cy) = \mathrm{span}\{\ksf(\Vx, \cdot)\}$, implies \cref{eq:2RKHSh2}.
\end{proof}

Henceforth, we restrict ourselves to radial kernels, 
for which a solid interpolation theory on submanifolds is available. 

\begin{assumption}\label{ass:radialkernel}
The kernel $\ksf$ satisfies $\ksf(\Vx,\Vy) = \Phi(\Vx-\Vy)$
for a positive definite function $\Phi\in L_2(\VR^d)\cap C(\VR^d)$.
Assume that the Fourier transform  $\hat \Phi$ of $\Phi$  satisfies
\begin{equation}\label{eq:fourier}
c_1 (1+\|\Vx\|^2)^{-\xi} \le  \hat \Phi(\Vx)  \le c_2 (1+\|\Vx\|^2)^{-\xi} \quad \text{ for all }\Vx\in \VR^2, 
\end{equation}
with $\xi >1$ and two constants $c_1,c_2>0$, $c_1\le c_2$. 
\end{assumption}
\begin{remark}
\cref{ass:radialkernel} implies that the native space of $\ksf$
on $\VR^2$ is $\Ch(\VR^2)=H^{\xi}(\VR^2)$ with equivalent norms. Thus, by the Sobolev trace theorem, the native space of $\ksf$ on $\partial\Omega$
is $\Ch(\partial\Omega)=H^{\xi-1/2}(\partial\Omega)$ \cite[Prop. 5]{FUWR12}.
\end{remark}

For such kernels, we can derive convergence rates
of $\Vr_\Vx^h$ to $\Vr_\Vx$ and $p_\Vx^h$ to $p_\Vx$ in certain Sobolev spaces.
For shortness, we focus on the more interesting term $\Vr_\Vx$
(convergence rates of $p_\Vx^h$ can be derived in exactly the same way).
Before stating the theorem, we recall the concept of fill distance of discrete sets.

\begin{definition}
The fill distance associated with a set $\mathcal X\subset \partial \Omega$ is defined by 
\begin{equation}
h_{\mathcal X,\partial \Omega} \coloneqq \adjustlimits\sup_{\Vy\in \partial \Omega}\inf_{\Vx\in \mathcal X} d_{\partial \Omega}(\Vx,\Vy),
\end{equation}
where $d_{\partial \Omega}$ denotes the geodesic distance associated with $\partial \Omega$.
Note that $d_{\partial \Omega}$ is equivalent to the Euclidean metric restricted on $\partial \Omega$
\cite[Theorem 6]{FUWR12}, that is, there are two constants $c_1,c_2>0$ (that depend only 
on $\partial \Omega$) such that
\ben\label{eq:distance}
c_1 \|\Vx-\Vy\| \le  d_{\partial \Omega}(\Vx,\Vy) \le c_2 \|\Vx-\Vy\|
\quad\text{ for all }  \Vx,\Vy\in \partial \Omega\,.
\een
\end{definition}

Given a continuous kernel $\ksf$ we define the integral operator $T:L_2(\partial \Omega) \to L_2(\partial \Omega)$ via
\ben
f \mapsto \left(Tf: \Vx \mapsto \int_{\partial \Omega}\ksf(\Vx,\Vs) f(\Vs)\; d\Vs\right)\,. 
\een
We define similarly the integral operator $\VT$ the acts on $L_2(\partial \Omega,\VR^2)$.

\begin{theorem}\label{thm:rxhrate}
Let $\Cy = \Cx$ and $\xi>1$, and $s = \xi - 1/2$. 
Let $\ksf$ satisfy Assumption~\ref{ass:radialkernel},
so that $\Ch(\partial \Omega) := H^s(\partial \Omega)$.
Assume that $T^{-1}(p_\Vx)\in L_2(\partial \Omega)$ and $\VT^{-1}(\Vr_\Vx)\in L_2(\partial \Omega,\VR^2)$. Then, there is a constant $c_{\partial\Omega}>0$, such that 
	\begin{equation}\label{eq:interpolation_rx}
	\|\Vr_\Vx^{h_{\mathcal X,\partial \Omega}} - \Vr_\Vx\|_{H^s(\partial \Omega)}
	\le c_{\partial\Omega} h_{\mathcal X,\partial \Omega}^{s}
	(\Vert \nubf\Vert_{C^k}^2\|\VT^{-1}\Vr_\Vx\|_{L_2(\partial \Omega,\VR^2)}
	+\Vert \taubf\Vert_{C^k}\|T^{-1}p_\Vx\|_{L_2(\partial \Omega)})\,
	\end{equation}	
	for all $h_{\mathcal X, \partial \Omega}\le h_{\partial \Omega}$.
\end{theorem}
\begin{proof}
In 	Corollary 15 from \cite{FUWR12} it is shown that there is a constant $c>0$, such that
\ben\label{eq:estimate_interpolate}
\|\Ci_{V_h}(\psi) -\psi\|_{\Ch} \le  c h_{\mathcal X, \partial \Omega}^{s}\|T^{-1}\psi\|_{L_2(\partial \Omega)} \quad \text{ for all } \psi \in \Ch(\partial \Omega)\,.
\een
A similar result holds for the operator $\VT$. Hence replacing
$\varphibf = \Ci_{[ V_h]^2}(\Vr_x)$ and $\psi=\Ci_{ V_h}(p_\Vx)$ 
in \cref{eq:r_x^hquasiopt} and using the estimate \cref{eq:estimate_interpolate}
gives the desired results.
\end{proof}

\begin{corollary}\label{cor:improved_rate}
Let the hypotheses of Theorem~\ref{thm:rxhrate} be satisfied. Then, there is a constant $c$ such that
\begin{equation}
\|\Vr_\Vx^{h_{\mathcal X,\partial \Omega}} - \Vr_\Vx\|_{L_2(\partial \Omega,\VR^2)} \le c h_{\mathcal X,\partial \Omega}^{2s}\,
\end{equation}	
for all $h_{\mathcal X, \partial \Omega}\le h_{\partial \Omega}$.
\end{corollary}
\begin{proof}
Applying \cite[Lem. 10]{FUWR12} yields
\ben
\|\varphibf - \Ci_{[V_h]^2}(\varphibf)\|_{L_2(\partial \Omega, \VR^2)} \le c h_{\mathcal X, \partial \Omega}^s\|\varphibf - \Ci_{[V_h]^2}(\varphibf)\|_{\Ch}
\een
for all $\varphibf \in [\Ch(\partial \Omega)]^2$
(and similarly for $\Ci_{V_h}$). Hence combining this inequality with \cref{eq:interpolation_rx} gives the desired estimate. 
\end{proof}

\begin{numexp}\label{numexp:rx_errors}
We perform a numerical experiment to verify the convergence rates derived in
\cref{thm:rxhrate}. We choose $\partial\Omega$ to be the ellipse depicted in
\cref{fig:plotrx} (left) and consider the sequence of subsets
\begin{equation}
\Cx_N = \{\gamma(2\pi\ell /N): \ell = 1,\dots, N\} \quad \text{with}\quad N=2^{4},2^{5},\dots,2^{10}\,.
\end{equation}
Then, for each $N$, we compute $\Vr^{h_N}_{\gamma(0)}$, where $h_N$ denotes the fill distance of $\Cx_N$.
Finally, we compute (with sufficiently many quadrature points) the $L_{2}(\partial\Omega)$-absolute error using $\Vr^{h_{2^{10}}}_{\gamma(0)}$
as reference solution. In \cref{fig:rx_errors}, we plot these $L_{2}$-errors versus $h_N$
(measured with the Euclidean distance) for the kernels $\ksf_4^{1.2}$, $\ksf_6^{1.2}$, and $\ksf_8^{1.2}$.
The measured convergence rates read (approximatively) 4.51, 6.79, and 8.77, respectively, and are 
slightly better but still in line
with \cref{cor:improved_rate}. Similar comparable results are obtained for the nonconvex domain
depicted in \cref{fig:plotrx} (left), as well as for different values of $\sigma$.
\end{numexp}
\begin{figure}[htb!]
\centering
\includegraphics{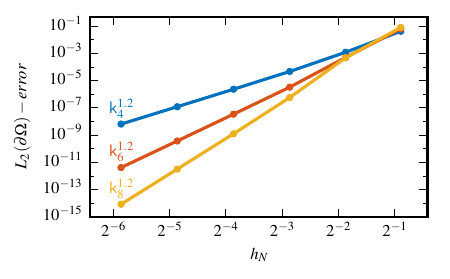}
\caption{\cref{numexp:rx_errors} confirms the convergence rates predicted by \cref{cor:improved_rate}:
the error $\|\Vr_\Vx^{h_N} - \Vr_\Vx\|_{L_{2}(\partial \Omega)}$ decays with the algebraic rate
4.51 (for $\ksf_4^{1.2}$), 6.79 (for $\ksf_6^{1.2}$), and 8.77 (for $\ksf_8^{1.2}$).}
\label{fig:rx_errors}
\end{figure}

\begin{remark}
It is known from scattered data approximation theory \cite{MR1325034} that the condition number of
$(\ksf(\Vx_i,\Vx_j))_{i,j=1}^{n}$ can be very large for certain collections of points $\{\Vx_i\}_{i=1}^n$.
We have computed the condition number of the saddle point matrix from \cref{eq:discretesaddlepoint2}
for different $\sigma$s and $h_N$ and have observed that
,similarly to the condition number of $(\ksf(\Vx_i,\Vx_j))_{i,j=1}^{n}$,
it behaves roughly as $(\sigma/ h_N)^p$,
where $p$ is the polynomial order of the Wendland's kernel.
{Note that by
combining Gershgorin's theorem \cite[Ch. 12, see p. 207]{Wendlandbook}
and the result listed in \cite[Ch. 12.2, Table 12.1]{Wendlandbook},
it is possible to show that the condition number of $(\ksf(\Vx_i,\Vx_j))_{i,j=1}^{n}$
is bounded from above by $C(h_N)^{-p}$, where $C$ is a positive constant.}
\end{remark}

\section{Application to shape Newton methods}\label{sec:Newton}
{The formulation of shape Newton methods is a delicate task, and
not all second-order shape derivatives are suited for this job.
The study of second-order shape derivatives is out of the scope of this work,
and we refer an interested reader to \cite{sturm3, Sc17}.}
In this section, we briefly recall the shape Newton method introduced in \cite{sturm3},
discuss its discretisation with the weakly-normal functions $\Vr_\Vx$s, and
perform some numerical simulations for an unconstrained shape optimisation test case.
This test case is kept simple on purpose in order to restrict the discretisation error
to the use of weakly-normal basis functions.
In this section, we use $\Cx=\Cy$ in \cref{def:rxh}.

\subsection{Shape derivatives and their structure theorem}
We begin by recalling the definition of shape derivatives
and their structure theorem from \cite{sturm3}.
For a given set $\Dsf\subset \VR^2$, we denote by $\wp(\Dsf)$ the powerset of $\Dsf$. 
{
Furthermore we use the notation $\ac C^1(\overbar\Dsf,\VR^d)$ to indicate the
Banach space of functions $\VX\in C^1(\overbar \Dsf,\VR^d)$ that satisfy $\VX=0$
on $\partial \Dsf$. }
\begin{definition}\label{def1} 
Let $J:\Xi\subset \wp(\Dsf)\rightarrow \VR$ be a shape function,
and let  $\VX,\VY \in \ac C^1(\overbar\Dsf ,\VR^2)$ be two vector fields.\\
(i)  The directional derivative of $J$ at $\Omega\in \Xi$ in the direction $\VX$ is defined by
\ben
DJ(\Omega)(\VX)\coloneqq \lim_{t \to 0^+}\frac{J((\id+t\VX)(\Omega))-J(\Omega)}{t}.
\een
(ii)  The second directional derivative of
$J$ at $\Omega$ in the direction $(\VX,\VY)$ is defined by
\ben
\FD^2J(\Omega)(\VX)(\VY) = \lim_{t \to 0^+}\frac{DJ((\id+t\VY)(\Omega))(\VX\circ (\id+t\VY)^{-1})-DJ(\Omega)(\VX)}{t},
\een      
\indent(assuming that $DJ((\id+t\VY)(\Omega))(\VX\circ (\id+t\VY)^{-1})$ exists for all small $t$).
\end{definition}
\begin{remark}%[Symmetry of $\FD^2J(\Omega)$]
{
Assume that the map
$J_\Omega: \ac C^1(\overbar\Dsf,\VR^d) \to \VR$, $\VX \mapsto %J_\Omega(\VX) := 
J((\id+\VX)(\Omega))$ % on $\ac C^1(\overbar\Dsf,\VR^d)$
is twice Fr\'echet differentiable at $0$.
Then, the bilinear operator $(\VX,\VY)\mapsto \FD^2J(\Omega)(\VX)(\VY)$
is the second derivative of $J_\Omega$ at $0$, and thus symmetric \cite[Thm 5.2, pp. 189]{b_AMES_II_2006a}.
}
\end{remark}
We denote by $\nabla^{\taubf} f$ the tangential gradient of function $f$ and by
$\VX_{\taubf}\coloneqq \VX|_{\partial \Omega}- (\VX|_{\partial \Omega}\cdot \nubf) \nubf$
the tangential component of a vector field $\VX$ on $\partial \Omega$.
The shape derivatives $DJ$ and $\FD^2J$ are characterised by the
following structure theorem (see \cite{sturm3,a_NOPI_2002a}).

\begin{theorem}\label{thm:structure_second}
Assume that $J$ is twice differentiable at $\Omega$ and assume that $\Omega$ is of class $C^3$.
Then, there are continuous functionals $\Fg: C^1(\partial \Omega) \to \VR$ and $\Fh:C^1(\partial \Omega)\times C^1(\partial \Omega) \to \VR$, such that
\begin{equation}\label{eq:struc_second_beforetwo}
DJ(\Omega)(\VX) = \Fg (\VX|_{\partial \Omega}\cdot \nubf)
\qquad\qquad \text{for all } \VX\in \ac C^1(\overbar\Dsf,\VR^2)\,,
\end{equation} 
and 
\begin{equation}\label{eq:struc_second_two}
\begin{split}
\mathfrak D^2J(\Omega)(\VX)(\VY) =&  \Fh(\VX_{|_{\partial \Omega}}\cdot 
\nubf, \VY_{|_{\partial \Omega}}\cdot \nubf) - 
\Fg(\nubf \cdot\partial^{\taubf}\VX_{\taubf}\VY_{\taubf})\\
%\Fg(\partial^{\taubf}\VX_{\taubf}\VY_{\taubf}\cdot \nubf )\\
&-	\Fg(\nabla^{\taubf}(\VY\cdot \nubf)\cdot \VX_{\taubf}) - \Fg(\nabla^{\taubf}(\VX\cdot \nubf)\cdot \VY_{\taubf})
\end{split}
\end{equation}
for all $\VX,\VY\in \ac C^2(\overbar\Dsf,\VR^2)$.
\end{theorem}

{
\begin{remark}
The function $\Fh$ corresponds to a Banach space version of the Riemannian Hessian 
on shape spaces used in \cite[Def. 2.2, p. 488]{Schulz1}.  
\end{remark}
}

\begin{remark}
	Notice that \cref{eq:struc_second_two} is equivalent to (see  \cite{a_NOPI_2002a})
	\ben
	\mathfrak D^2J(\Omega)(\VX)(\VY) =  \mathfrak h(\VX_{|_{\partial \Omega}}\cdot 
	\nubf, \VY_{|_{\partial \Omega}}\cdot \nubf) - \Fg( \nubf \cdot \partial^{\taubf}\VY \VX_{\taubf}) -  \Fg( \VY_{\taubf}\cdot \partial^{\taubf}\nubf \VX_{\taubf}) - \Fg(\nubf\cdot \partial^{\taubf}\VX \VY_{\taubf}).
	\een
\end{remark}

\subsection{Shape Newton descent directions and their approximation}
In this section, we define the shape Newton equation and its approximation
with weakly-normal basis functions. For a particular test case, we derive
quasi-optimality of the approximate Newton update and verify the resulting
convergence rates with a numerical experiment.
Henceforth, $\Omega\subset \VR^2$ is a set with $C^{k+1}$-boundary (with $k\ge 1$),
and $\mathcal H(\VR^2)$ is a RKHS that satisfies \cref{ass:one}.

\begin{definition}
The $\Ch(\VR^2)$-Newton descent direction at $\Omega$ is the solution
$\VX_{\nubf}\in [\mathcal H(\partial \Omega)]_{\nubf}$ of
 \begin{equation}\label{eq:newton_infinite}
 \FD^2J(\Omega)(\VX_{\nubf})(\VY) = - DJ(\Omega)(\VY) \quad \text{ for all } \VY\in [\Ch(\partial \Omega)]^2_{\nubf}.
 \end{equation}
\end{definition}
 
\begin{remark}
If $J$ satisfies the assumptions of \cref{thm:structure_second}
and {the functions in } $\mathcal H(\partial \Omega)$
{are sufficiently smooth (for instance, }
$\mathcal H(\partial \Omega)\subset C^2(\partial \Omega)$), then
{formulas \eqref{eq:struc_second_beforetwo} and \eqref{eq:struc_second_two} satisfy}
\begin{equation*}
\mathfrak D^2J(\Omega)(\Vr_\Vx)(\Vr_\Vy) = \Fh(\Vr_\Vx\cdot 
\nubf, \Vr_\Vy\cdot \nubf)
\quad \text{and} \quad 
DJ(\Omega)(\Vr_\Vy)  = \Fg(\Vr_\Vy\cdot \nubf)
\quad \text{for all }\Vx,\Vy\in \partial \Omega,
\end{equation*}
{because item \eqref{it:normal} from \cref{thm:propsofrx} 
implies $(\Vr_\Vx)_{\taubf} = (\Vr_\Vy)_{\taubf} = 0$.}
This, in light of \cref{thm:rxaredensesubset}, implies that \cref{eq:newton_infinite} is equivalent to
\ben\label{eq:NewtonBdry}
\Fh(\VX_{\nubf}\cdot 
\nubf, \Vr_\Vy\cdot \nubf) = - \Fg(\Vr_\Vy\cdot \nubf) \quad \text{ for all } \quad\Vy\in \partial \Omega. 
\een
Since $\Fh$ and $\Fg$ are (bi)linear and continuous in $C^1(\VR^2)$, equation \cref{eq:NewtonBdry}
can also be used to define $\Ch(\VR^2)$-Newton descent direction for RKHS that satisfy only
$\Ch(\VR^2)\subset C^1(\VR^2)$. For instance, as mentioned in \cref{ex:kernel}, the native space
of the kernel $\ksf_2^\sigma$ is $\Ch (\VR^2) = H^{2.5}(\VR^2)$, and
$H^{2.5}(\VR^2) \subset C^{1,1/2}(\VR^2)$ by the Sobolev embedding theorem
(where 
$C^{1,1/2}(\VR^2)$ denotes the usual H\"older space). 
\end{remark}

In general, it can be difficult to compute an \orange{exact} %explicit
solution of \cref{eq:newton_infinite}
because $\FD^2J$ and $DJ$ are typically given by very complicated formulas. For instance,
when $J$ is constrained by a PDE, the formula of $\FD^2J$
contains the material derivative of the solution of the PDE-constraint \cite{DeZo91}.
Nevertheless, approximate $\Ch(\VR^2)$-Newton descent directions can be computed
by discretising \cref{eq:newton_infinite} with the weakly-normal basis functions $\Vr_\Vx$s.

\begin{definition}
Let $\mathcal X_N=\{\Vx_1,\ldots, \Vx_N\}\subset \partial \Omega$
be a collection of pairwise distinct points, and let
$h = h_{\mathcal X_N,\partial \Omega}$ denote its fill distance.
The approximate $\Ch(\VR^2)$-Newton descent direction at $\Omega$ with respect to $\Ch(\partial \Omega)$
is the solution $\VX_{\nubf}^h\in \Cr_N\coloneqq\mathrm{span}\{\Vr_{\Vx}: \Vx \in \Cx_N\}$ of
\begin{equation}\label{eq:newton_discrete}
\mathfrak D^2J(\Omega)(\VX_{\nubf}^h)(\orange{\VY^h}) = - DJ(\Omega)(\orange{\VY^h}) \quad \text{ for all } \orange{\VY^h}\in \Cr_N.
\end{equation}
\end{definition}

\begin{example}\label{ex:unconstrained}
We consider the shape function $J(\Omega) = \int_{\Omega}f\; dx$ with $f\in C^2(\VR^2)$. This shape function $J$ can be used as a regularisation in image segmentation; see \cite{MR2049659}. Its shape derivatives read \cite{sturm3}
\begin{align}
\label{eq:dJlevelset}
DJ(\Omega)(\VX) = &\int_{\partial \Omega} f \VX\cdot \nubf\; ds\,,\\
\mathfrak D^2J(\Omega)(\VX)(\VY) =& \int_{\partial \Omega}  ( \nabla f\cdot \nubf + \kappa f) (\VX\cdot \nubf)(\VY\cdot \nubf)\; ds- \int_{\partial \Omega} f\nubf
\cdot \partial^{\taubf}\VX_{\taubf}\VY_{\taubf}\;ds \\
\nonumber
&-  \int_{\partial \Omega} f\nabla^{\taubf} (\VY\cdot \nubf)\cdot \VX + f\nabla^{\taubf} (\VX\cdot \nubf)\cdot \VY\; ds,
\end{align}
where $\kappa$ denotes the curvature of $\partial \Omega$. 
Hence the $\Ch(\VR^2)$-Newton descent direction at $\Omega$ is the solution
$\VX_{\nubf}\in [\mathcal H(\partial \Omega)]_{\nubf}$ of
\begin{equation}\label{eq:newton_infinite_example}
\int_{\partial \Omega}  ( \nabla f\cdot \nubf + \kappa f) (\orange{\VX}_{\nubf}\cdot \nubf)(\orange{\VY}\cdot \nubf)\; ds = - \int_{\partial \Omega} f (\orange{\VY}\cdot \nubf) \; ds \quad \text{ for all } \orange{\VY} \in [\mathcal H(\partial \Omega)]_{\nubf}.
\end{equation}
Note \orange{that}, %the
when it exists, the solution of \cref{eq:newton_infinite_example} is
$\VX_{\nubf} = (-f/(\nabla f\cdot \nubf+ \kappa f))\nubf$.
Clearly, it is possible to approximate $\VX_{\nubf}$
by interpolating it on $\Cr_N$. However, note that this
interpolant would not satisfy \cref{eq:newton_discrete}.
{ Finally, we point out that \cref{eq:dJlevelset}
implies that zero-level sets of $f$ are local optima for this shape optimization
problems. This implies that, close to the optimum, the term $\kappa f$
in \cref{eq:newton_infinite_example} plays a minor role compared to
$\nabla f\cdot \nubf$. For this reasons, 
some authors omit the term $\kappa f$ in the computation of Newton updates \cite{sturm3, Schulz1}.
}
\end{example}

\cref{prop:Xnuhquasioptimal} shows quasi-optimality of the approximate
shape Newton updates when the shape derivatives are as in \cref{ex:unconstrained}.
\begin{proposition}\label{prop:Xnuhquasioptimal}
Let $J(\Omega) = \int_{\Omega}f\; dx$ with $f\in C^2(\VR^2)$, and further assume that $\nabla f\cdot \nubf + \kappa f > 0$ on
$\partial \Omega$. Then, the solution
$\VX_{\nubf}^h\in \Cr_N$ of \cref{eq:newton_discrete}
satisfies
\begin{equation}\label{eq:approxerror_Xnu}
\Vert \VX_{\nubf} - \VX_{\nubf}^h\Vert_{L_2(\partial\Omega)} \leq c_{f,\Omega}
\inf_{\VY\in \Cr_N}
\Vert \VX_{\nubf} - \VY \Vert_{L_2(\partial\Omega)}\,.
\end{equation}
\end{proposition}

\begin{proof}
The bilinear form
\begin{equation}
L_2(\partial\Omega)\times L_2(\partial\Omega) \to \bbR\,,
\quad
(x,y)\mapsto \int_{\partial \Omega}  ( \nabla f\cdot \nubf + \kappa f) xy\; ds
\end{equation}
is $L_2(\partial\Omega)$-continuous,
symmetric, and $L_2(\partial\Omega)$-elliptic.
Let $\Cz_N\coloneqq\mathrm{span}\{\VX\cdot\nubf: \VX \in \Cr_N\}$.
Clearly, $\Cz_N\subset L_2(\partial\Omega)$.
Therefore, the solution $z_h\in\Cz_N$ of
\begin{equation}
\int_{\partial \Omega}  ( \nabla f\cdot \nubf + \kappa f) z_h y\; ds
=  - \int_{\partial \Omega} f y \; ds \quad \text{ for all } y \in\Cz_N \end{equation}
satisfies
\begin{equation}
\Vert z - z_h\Vert_{L_2(\partial\Omega)} \leq c_{f,\Omega}
\inf_{y\in \Cz_N}
\Vert z - y \Vert_{L_2(\partial\Omega)}\,,
\end{equation}
where $z\in L_2(\partial\Omega)$ is the solution of
\begin{equation}\label{eq:Newton_z}
\int_{\partial \Omega}  ( \nabla f\cdot \nubf + \kappa f) z y\; ds
=  - \int_{\partial \Omega} f y \; ds \quad \text{ for all } y \in L_2(\partial\Omega)\,.
\end{equation}
Finally, $\VX_{\nubf}^h = z_h\nubf$, because $\Cr_N$ satisfies
$\Cr_N = \{z\nubf: z \in \Cz_N\}$,
and $\VX_{\nubf}= z\nubf$, because the
solution of \cref{eq:Newton_z} is
$z=-f/(\nabla f\cdot \nubf+ \kappa f)$.
\end{proof}

\begin{numexp}\label{numexp:Xnu_errors}
We perform a numerical experiment to investigate
the approximation error \cref{eq:approxerror_Xnu} for
$f(\Vx) = \vert \Vx \vert^2-1$.
We choose $\partial\Omega$ to be the ellipse depicted in
\cref{fig:plotrx} (left) and consider the sequence of subsets
\begin{equation}
\Cx_N = \{\gamma(2\pi\ell /N): \ell = 1,\dots, N\} \quad \text{with}\quad N=2^{4},2^{5},\dots,2^{8}\,.
\end{equation}
For each $N$, we construct the finite dimensional space
$\widetilde \Cr_{N}\coloneqq\mathrm{span}\{\Vr^{h_N}_{\Vx}: \Vx \in \Cx_N\}$
(where $h_N$ denotes the fill distance of $\Cx_N$),
and compute the solution $\widetilde\VX_{\nubf}^{h_N} \in \widetilde \Cr_N$
of
\begin{equation}\label{eq:newton_finite_example}
\int_{\partial \Omega}  ( \nabla f\cdot \nubf + \kappa f) (\widetilde\VX_{\nubf}^{h_N}\cdot \nubf)(\orange{\VY^{h_N}}\cdot \nubf)\; ds = - \int_{\partial \Omega} f (\orange{\VY^{h_N}}\cdot \nubf) \; ds \quad \text{ for all } \orange{\VY^{h_N}} \in \widetilde \Cr_N.
\end{equation}
\orange{Although the normal vector field $\nubf$ and curvature $\kappa$ can be computed exactly,
for the ease of implementation we replace the curve $\gamma$ with a highly accurate interpolant
(using Chebfun \cite{chebfun}) and evaluate $\nubf$ and $\kappa$ of this interpolant. For a result on this kind of
approximation, see \cite{Su17}.}
Finally, we compute (with sufficiently many quadrature points) the $L_{2}(\partial\Omega)$-absolute error
$\Vert \VX_{\nubf} - \widetilde\VX_{\nubf}^{h_N}\Vert_{L_2(\partial\Omega)}$.
In \cref{fig:Xnu_errors}, we plot this $L_{2}$-errors versus $h_N$
(measured with the Euclidean distance) for the kernels
$\ksf_4^{0.7}$, $\ksf_6^{0.7}$, and $\ksf_8^{0.7}$.
The measured convergence rates read (approximatively) 4.20, 5.98,
and 8.20 (when $h_N>2^{-4}$), respectively. 
Similar compatible results are obtained for the nonconvex domain
depicted in \cref{fig:plotrx} (left), as well as for different values of $\sigma$.
To intepret these results, let us first point out that
$\widetilde \Cr_N$ is only an approximation of
$\Cr_N$.
It may not be so simple to prove a proposition for $\widetilde \Cr_N$
that is analogous to \cref{prop:Xnuhquasioptimal}, because $\widetilde \Cr_N$
does not contain the image of the operator $\VX \mapsto (\VX\cdot\nubf)\nubf$.
Note also that it is true that $\Vr_\Vx^{h_N}$ converges to $\Vr_\Vx$ as $h_N\to 0$,
but \cref{thm:rxhrate} does not guarantee that this convergence is uniform in $\Vx$.
However, if we assume that a proposition analogous to \cref{prop:Xnuhquasioptimal}
exists for $\widetilde \Cr_N$, we can explain the measured convergence rates with
\cite[Cor. 15]{FUWR12}, because interpolating a vector field $\Vf$
into $\widetilde \Cr_N$ or into $\Ch(\Cx_N)\times\Ch(\Cx_N)$
returns exactly the same interpolant if the vector field $\Vf$ is normal to $\partial\Omega$.

Finally, the convergence history saturates for $\ksf_8^{0.7}$ because the
condition number of the discretised shape Hessian is very large for $N=2^8$
(the \textsc{Matlab}-function \verb!cond! returns the value \verb!4.826490e+18!).
In a simple numerical experiment, we observed that the condition number of
the discretised shape Hessian behaves as $h_N^{\alpha_h}\sigma^{\alpha_\sigma}$
with $\alpha_h = -8.7$ and ${\alpha_\sigma} = 7.2$ when $p=4$ and
$\alpha_h = -12.48$ and ${\alpha_\sigma} = 10.18$ when $p=6$.
{We believe that these condition numbers deteriorate
because we rely on radial basis functions.
Indeed, following closely \cite[Sect. 5.2]{ChLe14},
it is possible to bound from above the condition number of a mass matrix the arises
from a discretization based on compactly supported Wendland kernels,
and this upper bound reads $C h_N^{-2p}$, where $C$ is a positive constant.}
\end{numexp}

\begin{figure}[htb!]
\centering
\includegraphics{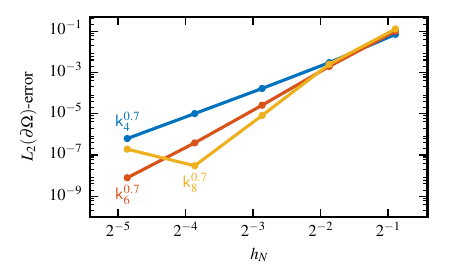}
\caption{
\cref{numexp:Xnu_errors} shows that
$\Vert \VX_{\nubf} - \widetilde\VX_{\nubf}^{h_N}\Vert_{L_2(\partial\Omega)}$
decays with the algebraic rate 4.20 (for $\ksf_4^{0.7}$), 5.98 (for $\ksf_6^{0.7}$),
and 8.20 (for $\ksf_8^{0.7}$ when $h_N>2^{-4}$).}
\label{fig:Xnu_errors}
\end{figure}

\subsection{Discrete shape Newton algorithm based on $\Vr_\Vx$}
In this section, we use the basis functions $\Vr_\Vx$s to formulate a discrete
shape Newton algorithm and test it on \cref{ex:unconstrained}.
In particular, we investigate the impact of the Newton update approximation
on the convergence rates of the shape Newton method.
To measure these convergence rates, we use the norm on the tangent spaces,
which is simply the norm of the RKHS.

Let $J$ be a twice differentiable shape function,
$\Omega\subset \VR^2$ be an initial guess, and $\Cx_N=\{\Vx_1,\ldots, \Vx_N\}$ be a
finite number points on $\partial\Omega$.
We begin by specifying how to transport the basis functions $\Vr_\Vx$
when $\Omega$ is perturbed by a geometric transformation.

\begin{definition}
Let $\Vr_{\partial \Omega,\Vx}$ denote the weakly-normal function associated with 
$\partial \Omega $ and $\Vx\in \partial \Omega$,
and let $F_t:\VR^2\to \VR^2$ be family of $C^1$-diffeomorphisms
such that $F_0=\id$. We define the transport $\FT_s$ by
\begin{equation}
\FT_s(\Vr_{\partial \Omega,\Vx}) \coloneqq \Vr_{F_s(\partial \Omega),F_s(\Vx)},
\end{equation}
and extend it by linearity to
$\mathfrak \FT_s:\mathrm{span}\{\Vr_{\partial \Omega,\Vx},\; \Vx\in \partial \Omega\} \to
\mathrm{span}\{\FT_s(\Vr_{\partial \Omega,\Vx}),\; \Vx\in \partial \Omega\}$.
By construction, $\mathfrak T_s:\mathrm{span}\{\Vr_{\partial \Omega,\Vx},\; \Vx\in \partial \Omega\}\subset [\Ch(\partial \Omega)]_{\nubf}^2 \to [\Ch(F_t(\partial \Omega))]_{\nubf}^2$. 
\end{definition}

Using this transport, we formulate the following discrete shape Newton method.
Set $\Omega_{0} = \Omega$, and
$\widetilde\Cr_N^0 \coloneqq \mathrm{span}\{\Vr^h_{\partial \Omega,\Vx}: \Vx\in \Cx_N\}$.
The optimisation algorithm proceeds as follows:
first, it computes the solution $\Vg_0\in \widetilde\Cr_N^0$ of
\begin{equation}\label{eq:newton}
\FD^2J(\Omega_0)(\Vg_0)(\varphibf) = - DJ(\Omega_0)(\varphibf) \quad \text{ for all } \varphibf\in \widetilde\Cr_N^0;
\end{equation}
then, it defines the transformation $F_0(\Vx) \coloneqq \Vx + \Vg_0(\Vx)$ and the new domain
$\Omega_1\coloneqq F_0(\Omega_0)$; finally, it defines the space $\widetilde\Cr_N^1$ by transporting $\widetilde\Cr_N^0$
with $F_0$, that is,
\begin{equation}
\widetilde\Cr_N^1 \coloneqq \mathrm{span}\{\Vr^h_{F_0(\partial \Omega),F_0(\Vx)}: \Vx\in \Cx_N\}\,.
\end{equation}
This procedure is repeated until convergence.

\begin{remark}
The shape $\Omega_1$ that results by updating $\Omega_0$ with $\Vg_0$
has the same regularity of $\Vg_0$. Therefore, from a theoretical point of view,
it may be necessary to regularise the update $\Vg_0$ to guarantee that
\cref{ass:one} remains fulfilled during the optimisation process. However,
such a regularization is not necessary in practice when
one employs the Wendland kernels from \cref{ex:kernel}. In this case,
regularity is lost only at the points $\{\Vx \pm \sigma : \Vx\in \Cx_N\}$,
whereas to construct the approximate weakly-normal function $\Vr_\Vx^h$s it
is necessary to evaluate $\nubf$ and $\taubf$ only at $\Cx_N$.
Therefore, if $\{\Vx \pm \sigma : \Vx\in \Cx_N\}\cap \Cx_N=\emptyset$,
the discrete algorithm does not perceive the loss of regularity of the domain.
\end{remark}

\begin{numexp}\label{numexp:Newton}
We test this discrete shape Newton method by performing 5 optimisation steps
to solve \cref{ex:unconstrained} with $f(\Vx) = \vert \Vx \vert^2 -1$ and
\begin{equation}
\Omega = \{ \gamma(\phi) = (0.2 + 1.15\cos(\phi), 0.15 + 0.9\sin(\phi)): \phi \in [0, 2\pi)\}\,.
\end{equation}
The (approximate) weakly-normal basis functions $\Vr^h_\Vx$ are constructed with the kernel $\ksf_4^{0.7}$
on the sets
\begin{equation}
\Cx_N = \{\gamma(2\pi\ell /N): \ell = 1,\dots, N\} \quad \text{with}\quad N=2^{4},2^{5},\dots,2^{8}\,.
\end{equation}

\orange{To simplify the implementation, the shape iterates
$\{\Omega_{\ell+1} = \Omega_\ell + \Vg_\ell(\Omega_\ell)\}_{\ell=0}^4$ are
replaced by highly-accurate interpolants computed with Chebfun \cite{chebfun}.
Normal and tangential vector fields $\nubf$ and $\taubf$, as well as the curvature term $\kappa$,
are approximated evaluating their corresponding analytic formulas
with these interpolants.}

%Following \cite{sturm3, Schulz1}, we neglect the curvature term in \cref{eq:newton_infinite_example}.
\orange{To replicate the results from \cite{sturm3, Schulz1},
we approximate Newton descent direction neglecting the curvature term in
\cref{eq:newton_infinite_example}. Note that a counterpart of \cref{prop:Xnuhquasioptimal}
can be proved, provided that $\nabla f\cdot \nubf >0$ (which is the case in these experiments).}
In \cref{fig:newton}, we plot the evolution of the shape iterates $\{\Omega_\ell\}_{\ell=0}^2$
for $N = 2^4$ (top left) and for $N = 2^5$ (top right),
and, for each $N$, we plot the values of the sequence
$\{\Vert \Vg_{\ell}\Vert_{[\Ch]^2}\}_{\ell=0}^4$ (bottom left)
and the measure of the quadratic rate of convergence 
$\{\Vert \Vg_{\ell+1}\Vert_{[\Ch]^2}/\Vert \Vg_{\ell}\Vert_{[\Ch]^2}^2\}_{\ell=0}^3$
(bottom right).

We observe that the sequence of shapes converges quickly to the minimiser
(a circle of radius 1 centered in the origin), and that the discretisation error
on the retrieved shape is not visible for $N\geq 2^5$.
We also observe that the discretised shape Newton method converges quadratically,
if sufficiently many weakly-normal vector fields are employed.
\end{numexp}
\begin{figure}[htb!]
\includegraphics[width=0.48\linewidth]{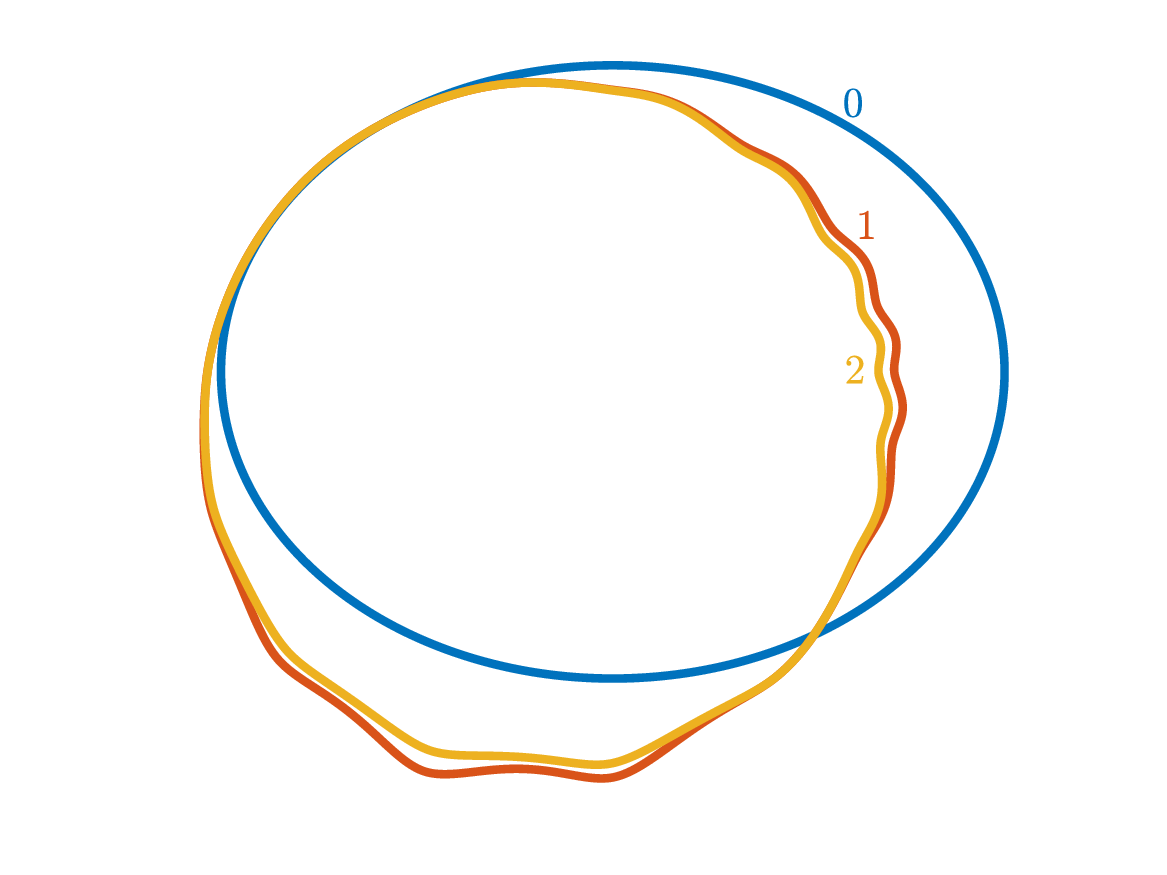}
\hfill
\includegraphics[width=0.48\linewidth]{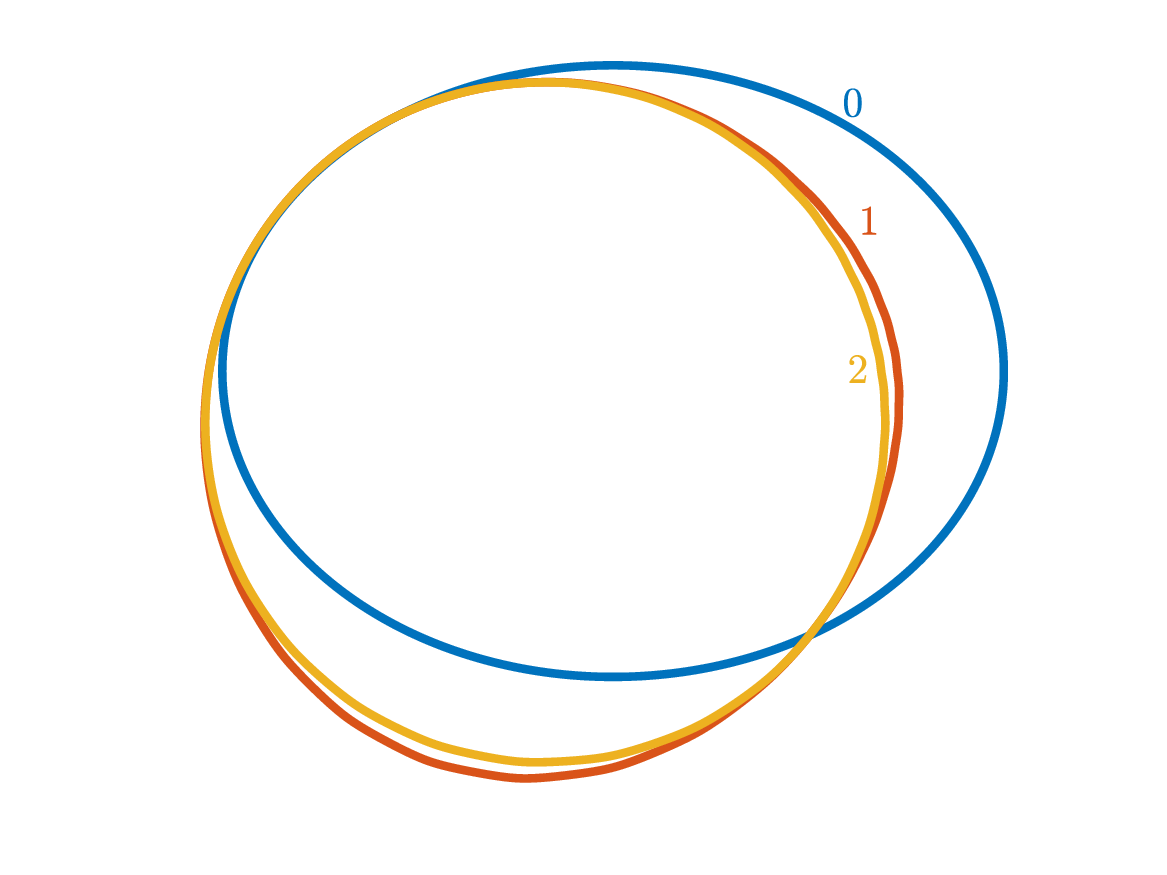}
\\
\includegraphics{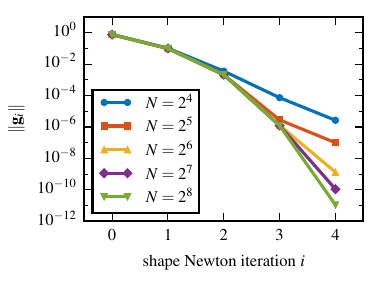}
\hfill
\includegraphics{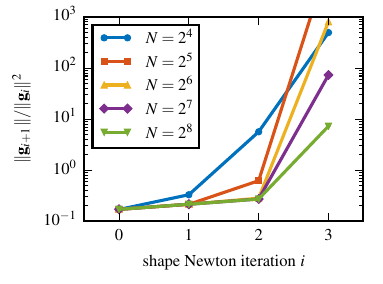}
\caption{\cref{numexp:Newton} shows that the discretised shape Newton method
converges quadratically, if sufficiently many weakly-normal vector fields are employed.}
\label{fig:newton}
\end{figure}
\begin{remark}
We have repeated the previous experiment including the curvature
term in the Newton equation. Surprisingly, we observed that the 
\orange{quadratic rate of convergence deteriorates to superlinear}
%convergence rates of the algorithm deteriorates drastically,
although the sequence of shape
iterates converge to the correct minimiser.
{To further investigate this unexpected behavior, we consider a
simpler scenario that can be solved analytically. Let the initial domain $\Omega_0$
be a disc of radius $r_0$ centered in the origin. In this case,
the shape iterates are concentric discs of radii
\begin{equation}\label{eq:concentric}
r_{\ell+1} = r_{\ell} -  \frac{r_{\ell}^2 -1}{2r_{\ell} + \kappa_\ell (r_\ell^2-1)}\,, \quad \ell = 0,1, \ldots, 
\end{equation}
where $\kappa_\ell = r_{\ell}^{-1}$ is the curvature of the $\ell$th disc. Equation
\eqref{eq:concentric} can be rewritten as
\begin{equation}\label{eq:r1}
  r_{\ell +1} = \frac{2r_{\ell}^3}{3 r_{\ell}^2 - 1}\,.
\end{equation}
From this equation, we see that if $r_{\ell}<1/\sqrt{3}$, then $r_{\ell +1}<0$, which is not feasible.
On the other hand, if $r_{\ell}>1/\sqrt{3}$, then $r_{\ell +1}>1$. Moreover, $r_{\ell+1}<r_\ell$ if $r_{\ell}>1$, which implies that $(r_\ell)$ decreases monotonically and {converges to one}.
By direct computation, we can show that
\ben\label{eq:rate1}
\vert 1 - r_{\ell +1} \vert = \left\vert\frac{(2r_{\ell}+1)(r_{\ell}-1)}{3r_{\ell}^2-1}\right\vert \vert 1-r_{\ell}\vert\,,
\een
which implies that the error is decreasing only if 
%$r_\ell > \fractx{1+\sqrt{41}}{10 }\approx 0.740312$ 
$r_\ell > \fractx{1+\sqrt{41}}{10} > \fractx{1+\sqrt{36}}{10} = 0.7$.
Since $r_{\ell}>1/\sqrt{3}$ implies $r_{\ell +1}>1$,
the sequence $\{r_\ell\}$ converges quadratically for every
initial radius $r_{0}>1/\sqrt{3}$ after the first iteration. 
%and it always converges quadratically after the first iteration.
However, we obtain a better algorithm if we set $\kappa_\ell = 0$, 
because equation \eqref{eq:concentric} becomes 
\begin{equation*}
  %r_{\ell+1} = \frac{r_\ell^2+1}{2r_\ell} = \frac12 (r_\ell + \frac{1}{r_\ell})\,,
  r_{\ell+1} = \frac12 \left(r_\ell + \frac{1}{r_\ell}\right)\,,
\end{equation*}
which converges quadratically for any $r_0>0$. To be more precise, in this case
\ben\label{eq:rate2}
|r_{\ell+1}-1| = \left\vert \frac{r_\ell-1}{2r_\ell}\right\vert|r_\ell-1|,
\een
which converges faster than
\eqref{eq:rate1} because $\fractx{2r_\ell +1}{3r_\ell^2-1}>\fractx{1}{2r_\ell}$ for
$r_\ell >\fractx{1}{\sqrt{3}}$.% is equivalent to $(r_\ell+1)^2 >0$. 
This example shows that the denominator in \eqref{eq:concentric} plays a key role in the
convergence analysis, and it may be for the same reason that the shape Newton algorithm 
without curvature term performed better in our previous numerical experiment.
}
\end{remark}
%\begin{remark}
%It is interesting to note that the shape Newton method applied with the initial shape as unit disc can be interpreted as a standard Newton method applied to the function $f(r):=\int_{B_r(0)} |x|^2-r_0\;dx$. In fact using polar coordinates we compute $f(r)=\fractx{\pi}{2}r_0 r^4-\pi r_0 r^2$ and hence $f'(r)=2\pi r^3-2\pi r_0 r$ and $f''(r) = 6\pi r^2 - 2\pi r_0$. Now it is readily checked that 
%\ben
%r_{\ell+1} = r_{\ell} - \frac{f'(r_\ell)}{f''(r_\ell)}, \quad\ell = 1,2,\ldots, 
%\een
%is indeed equivalent to \eqref{eq:concentric}.
%\end{remark}

{
\begin{numexp}\label{numexp:Newton_fcool}
To further highlight the effect of the discretisation parameter $N$
on approximate shape Newton methods, we consider the shape optimization problem
from \cref{ex:unconstrained} with $f(\Vx) = f(x,y) = (x^2+16y^2 -1)(x^2+(y-0.45)^2-0.04)$. The
zero-level set of this function is the union of a disc and an ellipse, and it has a cusp at
(0, 0.25) (see \cref{fig:newton_fcool}). Therefore, the optimal shape is not even Lipschitz,
and retrieving it by solving the optimization problem is particularly challenging.

We perform 15 discrete shape Newton steps
with the same numerical setup of \cref{numexp:Newton}
In \cref{fig:newton_fcool}, we plot the evolution of
shape iterates $\{\Omega_\ell\}_{\ell=0}^{15}$ for $N = 2^6$ (left) and for $N = 2^8$ (right).
These plots clearly show that increasing $N$ allows retrieving a much better approximation
of the optimal shape. Note that, after 15 steps, the algorithm has fully converged to an approximate optimum.

Note that these shape iterates have been computed neglecting the curvature term in
\cref{eq:newton_infinite_example}. Similarly to \cref{numexp:Newton}, we have
observed that the performance of shape Newton's method deteriorates if the curvature
term is taken into account. In particular, it is necessary both to regularized Newton's equation
to enforce coercivity \cite{HiRi03} and to switch to damped Newton's method
to ensure that every shape iterate is feasible. 

\end{numexp}
}

\begin{figure}[htb!]
\includegraphics[width=0.49\linewidth]{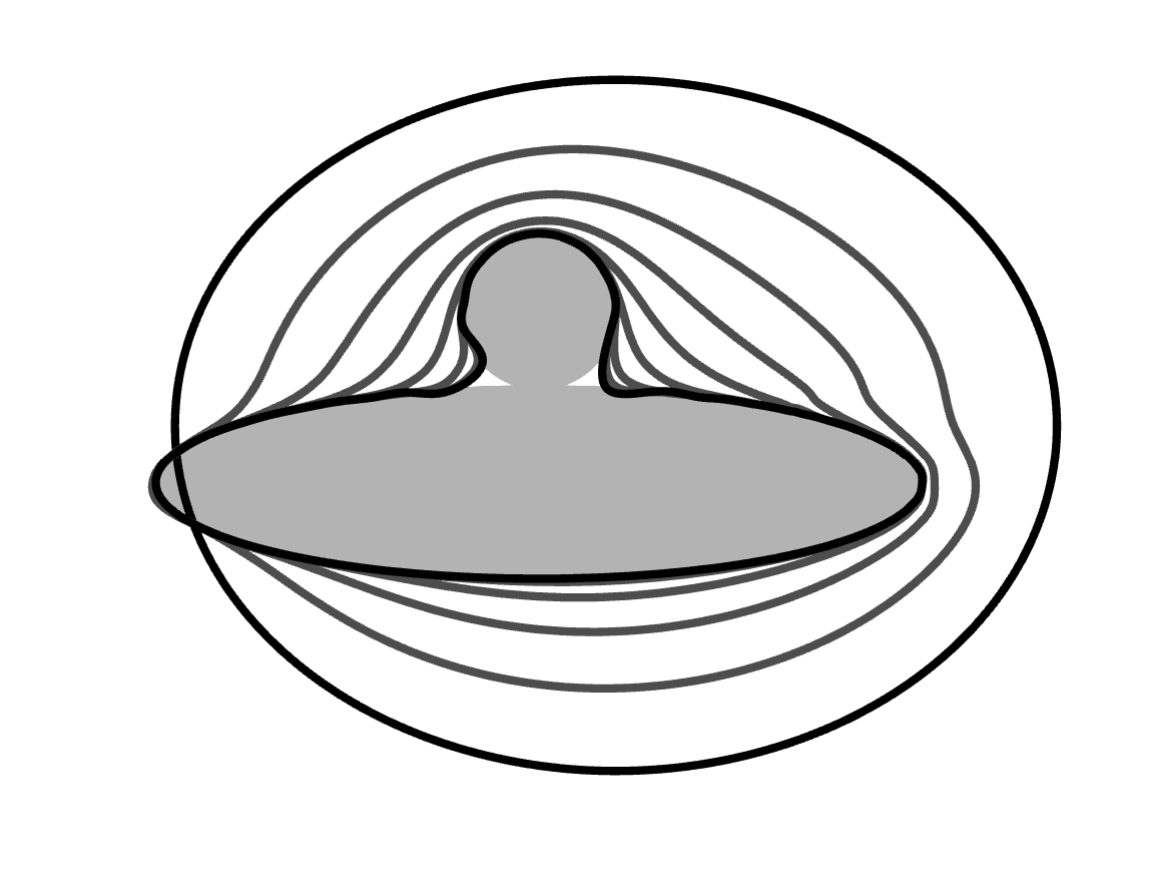}
\hfill
\includegraphics[width=0.49\linewidth]{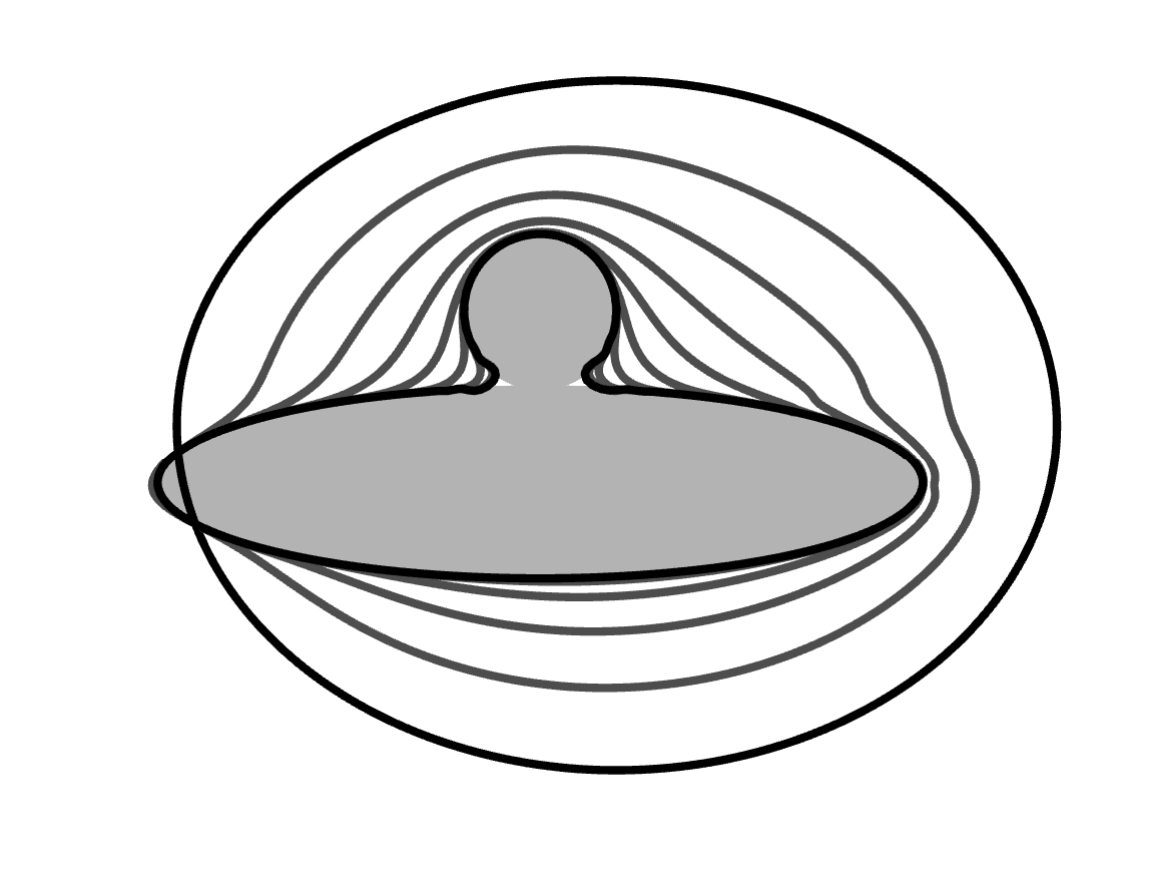}
\caption{\cref{numexp:Newton_fcool} shows that increasing the discretisation parameter
$N$ allows retrieving a much better approximation of the optimal shape. In these plots, the black lines
denote the initial and the final shape, whereas lighter lines denote intermediate shape iterates
and the gray area denotes the zero-level set of $f$.}
\label{fig:newton_fcool}
\end{figure}

\section{Conclusion}
In the first part of this work, we have introduced the class of
weakly-normal vector fields, which are defined as solutions of saddle point
variational problems in RKHSs.
Besides investigating their properties, we have discussed their
approximation and proved the related convergence rates.
These vector fields can be used to discretise
shape Newton methods because they have purely normal components.

In the second part of this work, we have formulated a discrete shape Newton
method based on these weakly-normal vector fields and, for a specific test
case, proved that the discrete shape Newton direction converges to the
continuous one in the metric associated with the shape Hessian.
Finally, we have showed that the discrete shape Newton method
exhibits quadratic convergence if the discretisation error
is sufficiently small.

{We have observed that our approach leads to
the matrices whose condition number deteriorates as the number of
degrees of freedom increases. This is a common issue of discretizations
based on radial basis functions. In \cite{ChLe14}, the authors describe
a multilevel residual correction algorithm that leads to approximation matrices
with bounded condition number. Although 
there is a debate on its convergence properties \cite{We18},
this algorithm performed well in numerical experiments, and it would
be interesting to investigate whether a similar technique could be used
in our context.
}

The weakly-normal vector fields have another property that has not
been exploited in this work: although defined on the boundary
$\partial \Omega$, they can be straightforwardly extended to $\VR^2$.
This implies that the discrete shape Newton method remains well-posed
even when formulated with volume-based expressions of the shape
derivative and the shape Hessian. This is a great advantage for
PDE-constrained shape optimisation problems,
in which case volume based formulas are easier to derive and impose
less regularity requirements on the solution of the PDE-constraint.
However, using volume-based formulas requires the integration of
the weakly-normal vector fields on $\Omega$, which introduces an
additional complexity in the implementation as well as additional sources
of discretisation error that need to be analysed.
For this reason, we postpone to a subsequent work the study of the shape
Newton method based on weakly-normal vector fields  in the framework of
PDE-constrained shape functionals.

{
Finally, we mention that certain shape optimization problems present an additional
difficulty: their second-order shape derivative is continuous
on a function space and coercive on a weaker one
%and coercive on different function spaces 
(this is commonly known as the ``norm gap''). When this is the case, the shape optimization
problem under consideration is not well posed \cite{EpHa12}.
Although discretisation is a form of regularization, in the case of an ill-posed problem
one should pay additional care to the discretisation level. The weakly-normal vector fields
introduced in this work may carry extra regularization properties because their smoothness
can be tuned by choosing appropriate reproducing kernels. We defer the investigation of this
interesting aspect to future work.
}

\section*{Appendix}
In this appendix, we show that $\mathcal H(\partial \Omega)\subset C^0(\partial \Omega)$
when $\ksf \in C^0(\VR^2\times\VR^2)$. We believe that this result is not new, but we
could not find its proof in the literature.
\begin{lemma}\label{lem:kernel_continous}
Let $\ksf$ be a symmetric positive-definite kernel on $\VR^2$. Denote by $\mathcal H(\partial \Omega)$ the RKHS associated with the restriction of $\ksf$ to $\partial \Omega$. If $\ksf \in C^0(\VR^2\times\VR^2)$,
then $\mathcal H(\partial \Omega)\subset C^0(\partial \Omega)$.\end{lemma}
\begin{proof}
For a generic $f\in \Ch (\partial \Omega)$,
the reproducing kernel properties of $\mathcal H(\partial \Omega)$
implies
\ben\label{eq:f_difference}
\begin{split}
|f(\Vx)-f(\Vy)| &= |(\ksf(\Vx,\cdot),f )_\Ch - (\ksf(\Vy,\cdot),f )_\Ch| \\
& = |(\ksf(\Vx,\cdot)-\ksf(\Vy,\cdot),f)_\Ch|  \\
& \le \|\ksf(\Vx,\cdot)-\ksf(\Vy,\cdot)\|_{\Ch}\|f\|_{\Ch}
\end{split}
\een
for all $\Vx,\Vy \in \partial \Omega$. Moreover,
\ben\label{eq:k_difference}
\|\ksf(\Vx,\cdot)-\ksf(\Vy,\cdot)\|_{\Ch}^2 = \ksf(\Vx,\Vx) + \ksf(\Vy,\Vy) - 2\ksf(\Vx,\Vy)\,.
\een
Therefore, the right hand side of \cref{eq:k_difference} goes to zero as $\Vx$ goes to $\Vy$
because $\ksf\in C^0(\VR^2\times\VR^2)$. In view of estimate \cref{eq:f_difference} and \cref{eq:distance},
this shows that $f$ is continuous on $\partial \Omega$.
\end{proof}

\bibliographystyle{siamplain}
\bibliography{refs}

\end{document}

%% file: slidesymbols.tex
\providecommand{\Supp}{\operatorname{supp}}                            % support
\providecommand{\supp}{\Supp}
\newcommand{\Va}{{\mathbf{a}}}
\newcommand{\Vb}{{\mathbf{b}}}

\newcommand{\Ve}{{\mathbf{e}}}
\newcommand{\Vf}{{\mathbf{f}}}
\newcommand{\Vg}{{\mathbf{g}}}

\newcommand{\Vp}{{\mathbf{p}}}

\newcommand{\Vr}{{\mathbf{r}}}
\newcommand{\Vs}{{\mathbf{s}}}

\newcommand{\Vv}{{\mathbf{v}}}

\newcommand{\Vx}{{\mathbf{x}}}
\newcommand{\Vy}{{\mathbf{y}}}
\newcommand{\Vz}{{\mathbf{z}}}

\newcommand{\VA}{{\mathbf{A}}}
\newcommand{\VB}{{\mathbf{B}}}

\newcommand{\VI}{{\mathbf{I}}}

\newcommand{\VL}{{\mathbf{L}}}

\newcommand{\VN}{{\mathbf{N}}}

\newcommand{\VP}{{\mathbf{P}}}

\newcommand{\VR}{{\mathbf{R}}}

\newcommand{\VT}{{\mathbf{T}}}

\newcommand{\VX}{{\mathbf{X}}}
\newcommand{\VY}{{\mathbf{Y}}}

\newcommand{\alphabf}{\boldsymbol{\alpha}}

\newcommand{\nubf}{\boldsymbol{\nu}}

\newcommand{\taubf}{\boldsymbol{\tau}}

\newcommand{\varphibf}{\boldsymbol{\varphi}}

\newcommand{\psibf}{\boldsymbol{\psi}}

\newcommand{\Dsf}{\mathsf{D}}

\newcommand{\ksf}{\mathsf{k}}

\providecommand{\Ch}{{\cal H}}
\providecommand{\Ci}{{\cal I}}

\providecommand{\Cl}{{\cal L}}

\providecommand{\Cn}{{\cal N}}

\providecommand{\Cr}{{\cal R}}

\providecommand{\Cx}{{\cal X}}
\providecommand{\Cy}{{\cal Y}}
\providecommand{\Cz}{{\cal Z}}

\providecommand{\bbR}{\mathbb{R}}

\providecommand{\FD}{\mathfrak{D}}

\providecommand{\FT}{\mathfrak{T}}

\providecommand{\Fg}{\mathfrak{g}}
\providecommand{\Fh}{\mathfrak{h}}

%% file: macrosKevin.tex
\newcommand{\overbar}[1]{\mkern 1.5mu\overline{\mkern-1.5mu#1\mkern-1.5mu}\mkern 1.5mu}

\DeclareMathOperator{\id}{\mathbf{id} }

\newcommand{\ac}{\accentset{\circ}}

\newcommand{\beqn}{\begin{eqnarray*}}
\newcommand{\eeqn}{\end{eqnarray*}}

\newcommand{\ben}{\begin{equation}}
\newcommand{\een}{\end{equation}}

\newcommand{\beq}{\begin{eqnarray}}
\newcommand{\eeq}{\end{eqnarray}}

\newcommand{\benn}{\begin{equation*}}
\newcommand{\eenn}{\end{equation*}}